\documentclass[letterpaper,11pt]{article}
\usepackage[letterpaper, top=1.1in, bottom=1.2in, left=1.1in,
  right=1.1in]{geometry}
\usepackage{mathtools}
\usepackage{amsmath,amssymb,amscd,fancybox}
\usepackage{ifthen,float,epsfig}

\usepackage{subcaption}
\usepackage{color}
\usepackage{comment}

\usepackage{graphicx}
\usepackage{amsfonts}
\usepackage{setspace}

\usepackage[T1]{fontenc}
\usepackage{mathpazo}

\usepackage[numbers,sort&compress]{natbib}

\usepackage{amsthm}
\usepackage{scalerel,stackengine}

\stackMath
\newcommand\reallywidehat[1]{%
\savestack{\tmpbox}{\stretchto{%
  \scaleto{%
    \scalerel*[\widthof{\ensuremath{#1}}]{\kern-.6pt\bigwedge\kern-.6pt}%
    {\rule[-\textheight/2]{1ex}{\textheight}}
  }{\textheight}%
}{0.5ex}}%
\stackon[1pt]{#1}{\tmpbox}%
}

\usepackage{titlesec}
\usepackage{lastpage}
\titleformat{\section}
  {\bf\large}{\thesection}{1em}{}
\titleformat{\subsection}
  {\normalfont\bf}{\thesubsection}{1em}{}
\titleformat{\subsubsection}
  {\normalfont\bf}{\thesubsubsection}{1em}{}

\newcommand{\er}{\mbox{\rm erf}}

\newcommand{\bx}{\boldsymbol{x}}
\newcommand{\bn}{\boldsymbol{n}}

\newtheorem{theorem}{Theorem}[section]
\newtheorem{lemma}[theorem]{Lemma}

\theoremstyle{definition}
\newtheorem{mydef}[theorem]{Definition}

\theoremstyle{remark}
\newtheorem{remark}{Remark}

\title{A new hybrid integral representation for frequency domain scattering
  in layered media} 

\numberwithin{equation}{section}

\author{Jun Lai\footnote{Courant Institute, New York University, New
    York, NY. Email: {\tt lai@cims.nyu.edu}}, \,  
Leslie Greengard\footnote{Courant Institute, New York University, New
  York, NY and Simons Foundation, New York, NY. Email:
  {\tt greengard@cims.nyu.edu}}, \, and 
Michael O'Neil\footnote{Courant Institute and School of
  Engineering,  New York University, New
    York, NY. Email: {\tt oneil@cims.nyu.edu}}
}

\date{\today}

\begin{document}

\maketitle 

\begin{abstract}
  A variety of problems in acoustic and electromagnetic
  scattering require the evaluation of impedance or layered media
  Green's functions. Given a point source located in an unbounded
  half-space or an infinitely extended layer, Sommerfeld and others
  showed that Fourier analysis combined with contour integration
  provides a systematic and broadly effective approach, leading to
  what is generally referred to as the Sommerfeld integral
  representation.  When either the source or target is at some
  distance from an infinite boundary, the number of degrees of freedom
  needed to resolve the scattering response is very modest.  When both
  are near an interface, however, the Sommerfeld integral involves a
  very large range of integration and its direct application
  becomes unwieldy.  Historically, three schemes have been employed to
  overcome this difficulty: the method of images, contour deformation,
  and asymptotic methods of various kinds. None of these methods make
  use of classical layer potentials in physical space, despite their
  advantages in terms of adaptive resolution
  and high-order accuracy. The reason for this is simple: layer
  potentials are impractical in layered media or half-space
  geometries since they require the discretization of an infinite
  boundary.
  In this paper, we propose a hybrid method which combines layer
  potentials (physical-space) on a finite portion of the interface
  together with a Sommerfeld-type (Fourier) correction.  We prove that
  our method is efficient and rapidly convergent for arbitrarily
  located sources and targets, and show that the scheme is
  particularly effective when solving scattering problems for objects
  which are close to the half-space boundary or even embedded across a
  layered media interface.
\end{abstract}

\onehalfspacing

\section{Introduction}
\label{sec_intro}

Problems of acoustic and electromagnetic wave scattering in half-space
or layered media geometries require the solution of the the governing
partial differential equations subject to suitable boundary and
radiation conditions.  In the two-dimensional, time-harmonic setting,
both reduce to the Helmholtz equation (Figure~\ref{figintro})
\begin{align}\label{helmholtz}
  \Delta u + k^2 u = f, 
\end{align}
with boundary conditions enforced on a scatterer $\Omega_0$ and interface
conditions enforced on the line $y=0$, either of
impedance (Robin) type
\begin{equation}
\label{impedance}
\frac{\partial u}{\partial n} +ik\alpha u = 0, 
 \end{equation}
or of transmission type
\begin{equation}
\label{continuity}
\left[u \right] = 0, \qquad \left[ \frac{\partial u}{\partial n}
  \right] =0.
\end{equation}
The Helmholtz coefficient $k$ is given as $k=\omega/c$, where $\omega$
is the governing frequency and $c$ is the wave speed in the medium.
For the sake of simplicity, we will assume that on the scatterer
$\Omega_0$ the total field $u$ satisfies homogeneous Dirichlet
boundary conditions
\[
u = 0 |_{\partial \Omega_0} .
\]

\begin{figure}[!t]
  \centering
  \includegraphics[width=.8\linewidth]{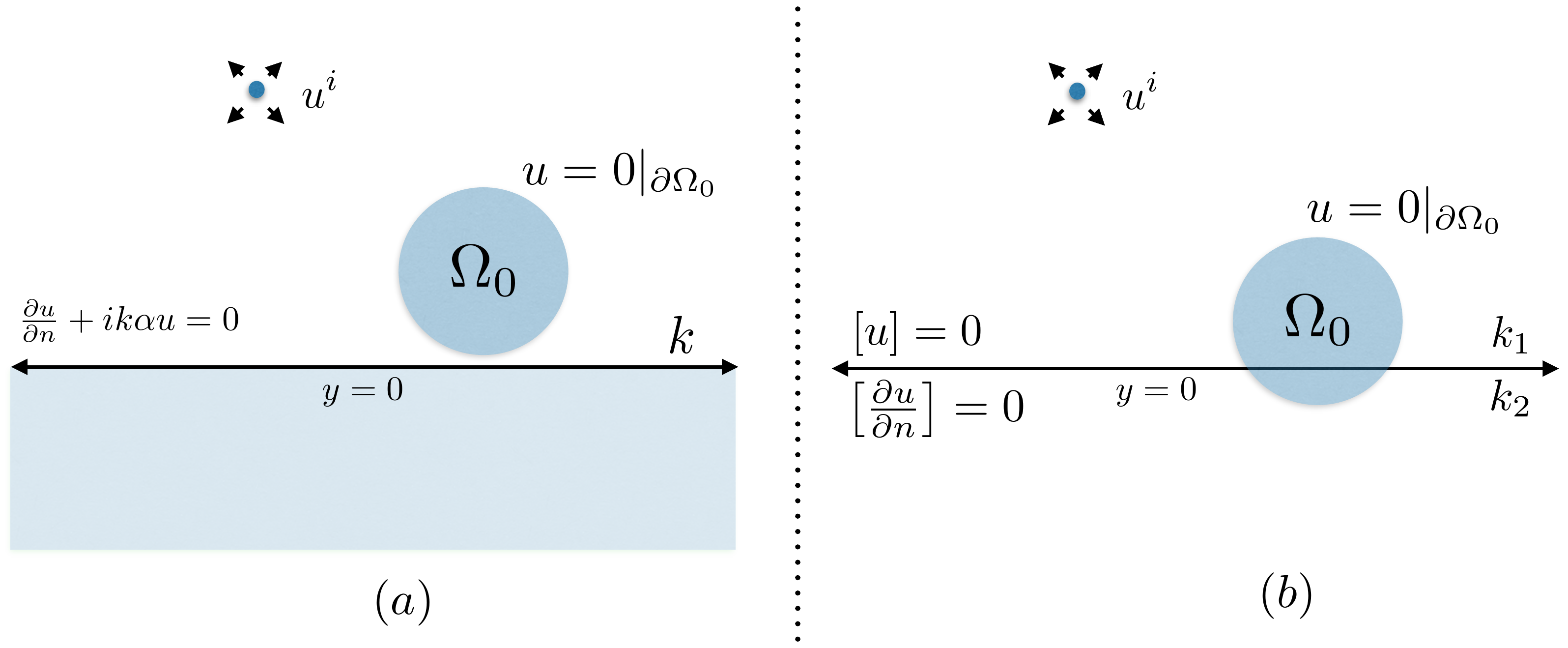}
  \caption{(a) Scattering in the presence of an impedance half-space,
    with a point source defining the incoming field and a sound soft
    scatterer $\Omega_0$. In (b), the impedance boundary is replaced
    by two-layer media, with a distinct Helmholtz parameter in the
    lower half-space. The scatter $\Omega_0$ is partially buried and
    touches both media.
  \label{figintro}}
\end{figure}

In electromagnetics, this condition corresponds to the case of
scattering from a perfectly conducting obstacle in transverse-magnetic
(TM) polarization, and in acoustics to the case of a sound-soft
obstacle.  Here and in what follows, $\bn=(0,1)$ is the unit normal on
the line $y=0$, $\frac{\partial u}{\partial n}$ denotes the partial
derivative of $u$ in the normal direction, and $\alpha$ is an
impedance constant~\cite{chandler-wilde} with $\Re(\alpha)\ge 0$.  The
expression $[f]$ denotes the jump in the function $f$ across the line
$y=0$, which we will denote by $\Gamma$ in the remainder of the paper.
We will denote by $u^i$ the incoming field induced by the sources $f$
in (\ref{helmholtz}). We will limit our attention, without loss of
generality, to either point sources or plane waves.  To ensure
uniqueness of the boundary value problem, a radiation condition must
be imposed to enforce that the scattered field is decaying.  Thus, we
assume that the total field $u$ is written in the form $u = u^i +
u^s$, where the scattered field $u^s$ satisfies the Sommerfeld
radiation condition \cite{Cot2}:
\begin{equation}\label{rad0}
\lim_{r\rightarrow \infty}\sqrt{r}\left( \frac{\partial u^s}{\partial
  r}-iku^s\right)=0.
\end{equation}
We will also assume that the Helmholtz parameter $k$ is constant in
either the upper or lower half-space, with $\Re(k)>0$ and $\Im(k) \ge
0$.  Some applications require variants of the interface conditions
above, such as
\begin{equation}
\left[\gamma u \right] = f, \qquad \left[ \beta \, \frac{\partial
    u}{\partial n} \right] =g,
\end{equation}
where $\gamma, \beta$ are piecewise constant material
parameters~\cite{chew,lindell1, lindell2, lindell3}.  The method of
this paper extends to these cases in a straightforward manner.

Integral equations are natural candidates for solving the problems
described above since they discretize the scatterer alone and
impose the Sommerfeld radiation condition by construction.  In order
to make effective use of this approach, however, one must generally
evaluate the governing Green's function which satisfies the
homogeneous interface conditions (\ref{impedance}) or
(\ref{continuity}). This avoids the need to discretize the interface
$\Gamma$ (the infinite line $y=0$), and the most common treatment,
using Fourier analysis, was pioneered by Sommerfeld, Weyl and
Van~der~Pol~\cite{sommerfeld,weyl,vanderpol}.  Over the past several
decades, a number of methods have been developed, based on this
approach.  These include using ideas from high-frequency asymptotics,
rational approximation, contour deformation
~\cite{cai,caiyu,okhmatovski-2014,paulus_2000,Perez-Arancibia2014}, 
complex images~\cite{thomson-1975,ochmann,taraldsen}, and methods based on
special functions \cite{koh} or physical images \cite{oneil-imped}.

Without reviewing the methods listed above in detail, we simply note
that all of them are aimed at the efficient pointwise evaluation of
the impedance or layered media Green's function, rather than treating
the scattering problem itself in a more unified fashion.  Here, we
consider a substantially different approach, motivated by the fact
that close-to-touching interactions between compactly supported
scatterers are easily accounted for using standard {\em physical
space} layer potentials, which can be adaptively refined to cope
with the near singularities induced by the geometry (see, for example,
\cite{greengard_helsing,hao_barnett,helsing_1996}).  Potential theory
cannot be used naively in layered media, however, because of the
infinite extent of the interface. In principle, however, it seems
plausible that the failure of rapid convergence of the Sommerfeld
integral is due entirely to the close-to-touching interaction. This
turns out to be the case, and we present a rigorous, hybrid method for
scattering problems that combines the best features of layer
potentials (adaptivity with high order convergence) and of the
Sommerfeld representation (spectral accuracy for smooth functions).
Put differently, layer potentials in physical space will be used to
capture features of the solution with high-frequency components in the
Fourier domain, and the Sommerfeld integral will be responsible only
for the remaining low-frequency components. For this we solve a local
integral equation and apply a smooth window function to the solution
in order to capture the most singular behavior.

\begin{remark}
It is worth emphasizing that window functions have been used previously
to accelerate the evaluation of layered media Green's functions
~\cite{cai,caiyu,paulus_2000,Perez-Arancibia2014}.
The approach in those papers, however, is based on using carefully chosen 
partitions of unity and asymptotic analysis to evaluate 
slowly decaying and oscillatory Sommerfeld integrals for each source and
target. We are using a partition of unity in {\em physical} space to enforce 
rapid decay in a single Sommerfeld-type correction that can be used for all 
target points.

Bruno and collaborators \cite{bruno_lecture_2015,bruno_window_2015} 
have independently developed a method
that makes use of the same underlying intuition. They also 
propose using a window function regularizing the integral operator
in physical space in order to handle the most complicated features of the 
scattering problem.  Unlike our scheme,
asymptotic methods and stationary phase arguments are used to approximate 
the solution of the integral equation on the entire interface.
Their windowing scheme 
could be adapted for use in place of our local integral 
equation (described below), with the potential for eliminating the artificial
endpoint singularities introduced in our scheme. This
would reduce the complexity of the implementation  \cite{Bruno_personal}. 
\end{remark}

An outline of the paper is as follows: 
in Section~\ref{phys_spec}, we review some basic properties of the 
free-space Green's function, layer potentials, and their spectral 
representation.
In Sections~\ref{sec_impedance} and~\ref{sec_layered}, 
we focus on the construction of a {\em local} integral equation
and describe our hybrid representation  in more detail.
We then prove, in Section~\ref{sec_wellposed}, 
that our local integral equation is well-posed.
In Section~\ref{sec_sommerfeld}, we prove decay estimates on the 
integrand of our Sommerfeld correction.
Section~\ref{sec_numerical} contains some illustrative examples 
for both pointwise evaluation of the Green's function and for scattering 
computations in the presence of obstacles. 
Section~\ref{sec_conclusions} contains some brief, concluding remarks.

\section{Spectral representation of the Green's function} \label{phys_spec}

For $k \in \mathbb C$ with non-negative imaginary part,
it is well-known that the
Green's function for the free-space Helmholtz
equation~\eqref{helmholtz} is the zeroth order Hankel function of the
first kind:
\begin{equation}
G_k(\bx,\bx_0) = \frac{i}{4}H_0^{(1)}(k|\bx-\bx_0|).
\end{equation} 
It satisfies 
\begin{equation}\label{eq_helm}
(\Delta + k^2) G_k(\bx,\bx_0) = \delta(\bx - \bx_0),
\end{equation}
where $\delta(\bx-\bx_0)$ represents the delta function at $\bx_0$.
The Sommerfeld integral representation for $H_0^{(1)}$ is
obtained by taking the Fourier transform of equation~(\ref{eq_helm})
with Fourier coordinates $(\lambda_x,\lambda_y)$, and evaluating the
integral in $\lambda_y$ via contour integration
\cite{sommerfeld,weyl,vanderpol}. This yields
\begin{equation}\label{eq_sommerfeld}
  G_k(\bx,\bx_0) = \frac{1}{4\pi} \int_{-\infty}^\infty
  \frac{e^{-\sqrt{\lambda^2-k^2}|y-y_0|}}{\sqrt{\lambda^2-k^2}} \, 
e^{i\lambda (x-x_0)} \, d\lambda.
\end{equation}
Since $\lambda_x$ is the only Fourier variable remaining after contour
integration, we denote it simply by~$\lambda$.

The formula (\ref{eq_sommerfeld}) plays a central role in scattering theory,
and numerous schemes are available that yield high order accuracy upon
discretization. The simplest, perhaps, is contour deformation to avoid 
the square root singularity in the integrand (see for example, 
\cite{barnett_greengard,chew}).
We do not review the literature here, since we are primarily concerned with 
the range of integration required in the 
Fourier integral parameter $\lambda$. The relevant considerations are
most easily understood in the context of computing the Green's function
for a perfectly conducting or sound-soft half-space, where we seek to 
impose the Dirichlet condition $u=0$ on the interface~$\Gamma$.

The obvious solution, of course, is to construct the corresponding
Green's function, which we denote by $G^0_k$, using the method of
images. Assuming $\bx_0 = (x_0,y_0)$ with $y_0>0$, let its reflected
image $\bx^R_0 = (x_0,-y_0)$. It is then easy to verify that
\begin{equation}
\label{Gk0def}
G^0_k(\bx,\bx_0) = \frac{i}{4}H_0^{(1)}(k|\bx-\bx_0|) - 
\frac{i}{4}H_0^{(1)}(k|\bx-\bx^R_0|).
\end{equation} 
Using the Sommerfeld integral, we may instead write 
\begin{equation}
\label{Gk0Sommerfeld}
G^0_k(\bx,\bx_0) = \frac{i}{4}H_0^{(1)}(k|\bx-\bx_0|) - 
  \frac{1}{4\pi} \int_{-\infty}^\infty
  \frac{e^{-\sqrt{\lambda^2-k^2}|y+y_0|}}{\sqrt{\lambda^2-k^2}} \, 
e^{i\lambda (x-x_0)} \, d\lambda.
\end{equation} 
While this is more complicated than (\ref{Gk0def}), the analogous
approach can be used even when the boundary or interface condition
does not support a simple image representation, as we shall see below.
For the moment, we simply wish to observe that when $|y+y_0|$ is
large, the integrand in (\ref{Gk0Sommerfeld}) is rapidly decaying,
once $\Re(\lambda^2 - k^2) >0$. When $|y+y_0|$ is small, however, the
decay is negligible and the range required in $\lambda$ is large. In
this regime, the Sommerfeld representation is extremely inefficient
unless various asymptotic or contour deformation methods are employed.

A third approach to imposing the homogeneous Dirichlet interface
condition $u=0$ on $\Gamma$ is to write $u(\bx) =
\frac{i}{4}H_0^{(1)}(k|\bx-\bx_0|)+ u^s(\bx)$, with the scattered field
$u^s$ represented as a double layer potential with unknown density
$\sigma$:
\begin{equation}
    u^s(\bx) = D_\Gamma^k[\sigma](\bx) 
    \equiv \int_\Gamma \left[ \frac{\partial}{\partial n'}
  G_k(\bx,\bx')
  \right] \, \sigma(\bx') \, ds(\bx').
\end{equation}
Here $s$ denotes arclength along $\Gamma$, and $\partial/\partial n'$
denotes differentiation in the outward normal direction at the point
$\bx'$. Substituting the representation of $u^s$ into the boundary
condition $u=0$ yields a second-kind integral equation:
\begin{equation}\label{eq_doublef}
\frac{\sigma}{2} + D_\Gamma^{k \ast}[\sigma] = - u^i,
\end{equation}
where $D^{k \ast}_\Gamma$ denotes the principal value of $D^k_\Gamma$.
On a half-space, however, it is well-known that $D^{k \ast}_\Gamma$
vanishes \cite{guenther_lee}, so that $\sigma = -2u^i$ and
\[
u^s(\bx) = D_\Gamma^k\left[ -2u^i \right](\bx) \, .
\]
This is not useful in practice, because $\Gamma = (-\infty,\infty)$
and $u^i$ is slowly decaying in space, so that the range of
integration is extremely large (see~Figure~\ref{figdirdensity}). Note,
however, that when $y_0$ is small, $u^i$ is nearly singular only at
the close to touching point in physical space. This accounts for the
slow convergence of the integrand in the Sommerfeld integral,
$\frac{e^{-\sqrt{\lambda^2-k^2}y_0}}{\sqrt{\lambda^2-k^2}} \,
e^{-i\lambda x_0}$, which is simply the Fourier transform of $u^i$
along the line $\Gamma$.

\begin{figure}[t]
  \centering
  \includegraphics[width=.8\linewidth]{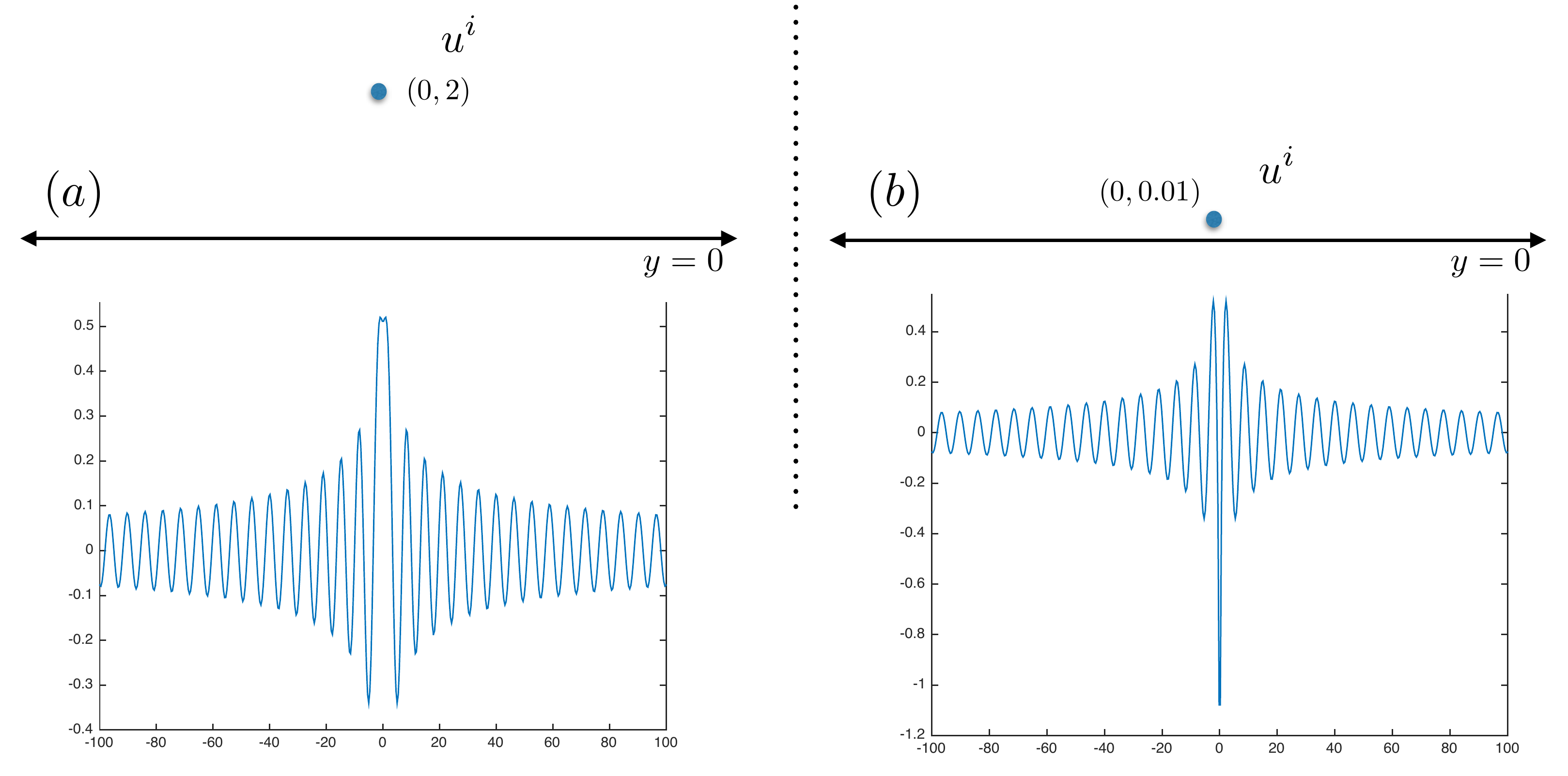}
  \caption{(a) For a point source at $(0,2)$ with $k=1$, we plot the
    imaginary part of the function $u^i$ on $\Gamma$.  The response is
    smooth, oscillates over the indicated range $(-100,100)$, and
    decays slowly (like $1/\sqrt{x}$). (b) When the point source is
    near $\Gamma$, the oscillations and slow decay are still present,
    but the function is nearly singular at the close-to-touching
    point.
\label{figdirdensity}}
\end{figure}

Our hybrid approach is based on the following premise: that the poor
convergence of the Sommerfeld integral is due entirely to the near
singularity in the layer potential density \mbox{($\sigma = -2 u^i$)}
at the close-to-touching point. Thus, we partition $u^i$ into a local
part, for which we may make effective use of layer potentials, and a
remainder, for which we can effectively use the Sommerfeld
representation (see Figure~\ref{figdirwindow}).

\begin{figure}[t]
  \centering
  \includegraphics[width=.8\linewidth]{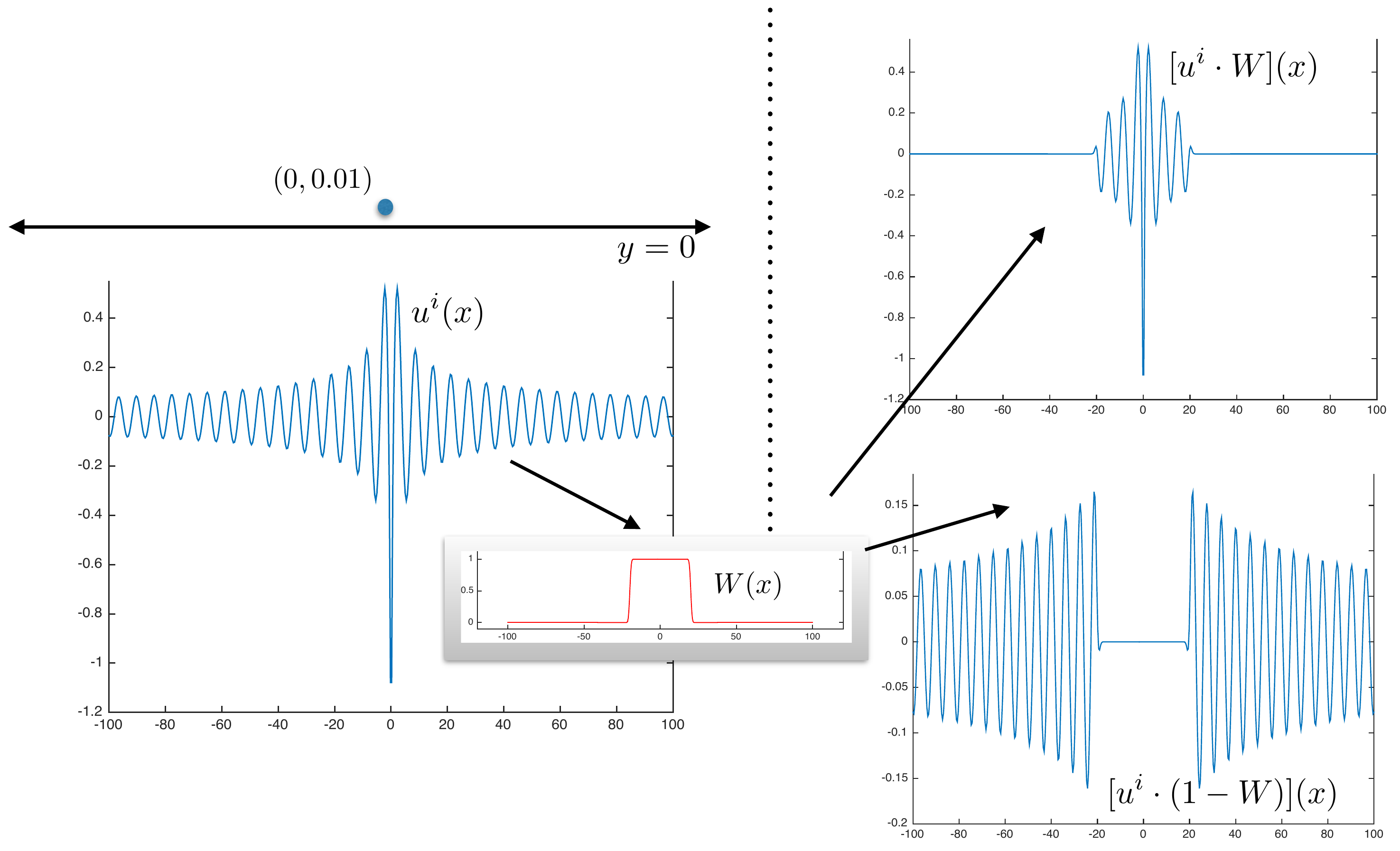}
  \caption{For a point source near the boundary, the double layer
    density $2u^i$ is nearly singular on $\Gamma$. We use a smooth
    window function (see text) to capture the nearly singular part in
    physical space (top right), and the Sommerfeld integral to account
    for the remainder (bottom right).
\label{figdirwindow}}
\end{figure}

More precisely, let us assume we have a window function $W(x)\in
C^{\infty}(\mathbb{R})$, supported on a finite section of the
interface, $\Gamma_0 = (-M_0,M_0)$, that satisfies
\begin{equation}\label{bump}
W(x)=
\begin{cases}
0 \quad \mbox{ for } x\le -3/4M_0,  \\
1 \quad \mbox{ for } -1/4M_0 \le x\le 1/4M_0, \\
0 \quad \mbox{ for } x\ge 3/4M_0.
\end{cases}
\end{equation} 
We can construct such a window function from a compactly supported
$C^\infty$ function, such as the standard {\em bump} function: 
\[ \Psi(x) = \left\{ \begin{array}{rl}
e^{-\frac{1}{1-x^2}} & \mbox{for $|x| < 1$}; \\
0 & \mbox{otherwise}.
\end{array} \right.
\]
For this, we let 
$\Phi(x) = \frac{1}{A} \int_0^x \Psi(t) \, dt$, where
$A = \int_{0}^1 \Psi(t) \, dt$, and define
\begin{equation}\label{cinfbump}
  W(x) = W_{M_0}(x) =  \frac{1}{2} \left(\Phi(x + M_0/2)
  - \Phi(x - M_0/2)\right) \, .
\end{equation}
It is straightforward to verify that, once $M_0 > 4$, $W(x)$ satisfies
the desired conditions. We use notations $W$ and $W_{M_0}$
interchangeably, depending on the particular relevance of the
parameter~$M_0$.

\begin{remark}
In practice, we use 
\begin{equation}\label{err}
\tilde{W}(x) = 
\tilde{W}_{M_0}(x) =  1/2\left(\er(x+M_0/2) - \er(x-M_0/2)\right) 
\end{equation}
where $\er(x)$ is the error function 
\begin{equation}
\er(x) = \frac{2}{\sqrt{\pi}}\int_{0}^x e^{-t^2} \, dt.
\end{equation}
When $M_0 = 28$, $\tilde{W}$ in \eqref{err} satisfies
\eqref{bump} with more than fourteen digits of accuracy.
For smaller windows (smaller values of 
$M_0$), $\tilde{W}$ can simply be rescaled. The smaller the window,
however, the less rapid the 
decay of its Fourier transform. 
\end{remark}

Letting $\sigma_W(x) = W(x) \cdot \sigma(x) = W_{M_0}(x) \cdot
(-2u^i(x))$, we seek to represent the scattered field in the form
\begin{equation} \label{hybriddir}
u^s(\bx) = 
    \int_{\Gamma_0} \left[ \frac{\partial}{\partial n'}
  G_k(\bx,\bx')
  \right] \, \sigma_W(\bx') \, ds(\bx')
+   
  \frac{1}{4\pi} \int_{-\infty}^\infty
  e^{-\sqrt{\lambda^2-k^2}y} \, 
e^{i\lambda x} \, \hat\xi_W(\lambda) \, d\lambda.
\end{equation}
Using the decomposition $\sigma = \sigma_W + (1-W) \sigma$, 
it is straightforward to verify that 
\begin{equation}
 \hat\xi_W(\lambda) = -2 \, \widehat{(1 -W)} \ast \left(
 \frac{e^{-\sqrt{\lambda^2-k^2}y_0}}{\sqrt{\lambda^2-k^2}}
 e^{-i\lambda x_0} \right),
\label{sommerfeld_cor}
\end{equation}
where $\widehat{f}(\lambda)$ denotes the Fourier transform of $f(x)$:
\[
\widehat{f}(\lambda) = \int_{-\infty}^{\infty} f(x) \, e^{-i\lambda x} \,
dx.
\]
That is, $u^s$ is represented as a double layer potential over a
finite region $\Gamma_0$ plus a Sommerfeld correction with density
$\hat\xi_W$, defined in (\ref{sommerfeld_cor}). It turns out that
$\hat\xi_W$ is rapidly decaying as a function of $\lambda$, as shown
in Theorem \ref{dirthm} below.  In subsequent sections, we show how
this hybrid representation can be applied to the cases of interest,
rather than merely the Dirichlet problem where one would (of course) use the
simple image solution in practice.

\begin{lemma}\label{lema1}
The Fourier transform of the window function $W_{M_0}$ in 
(\ref{err}) decays superalgebraically.
\end{lemma}
\begin{proof}
This follows immediately from the fact that $W$ is $C^\infty$ and
integration by parts.
\end{proof}

\begin{theorem} \label{dirthm}
Let $u^s$ be the scattered field induced by a point source at $(0,h)$
satisfying the Dirichlet boundary condition $u^s = -u^{i}$. If we
represent $u^s$ in the form (\ref{hybriddir}), then 
$|\hat{\xi}_W(\lambda)|$ decays superalgebraically,
independent of $h$.
\end{theorem}
\begin{proof}
Rather than deriving this estimate directly from formula 
(\ref{sommerfeld_cor}), we write
\[  
\hat{\xi}_W(\lambda) = -2  \left[
    \reallywidehat{1 - W_{M_0}} \ast \reallywidehat{u^i}  \right]= 
-2 \left[  \reallywidehat{1 - W_{M_0}} \ast (\reallywidehat{u^i}
  -\reallywidehat{u_{M_0}^i} +
  \reallywidehat{u_{M_0}^i}) \right],
\]
where $u_M^i$ denotes the field due to a 
smooth compactly-supported distribution centered at the source point $(0,h)$ 
of the form
\[ \frac{1}{B} W (3 |\bx-(0,h)|) \, .
\]
This scaled version of $W$ vanishes at a distance $M_0/4$ from the center.
If we let 
\[ B = 2 \pi \int_0^{M_0/4} J_0(k \rho) \, W( 3 |\bx-(0,h)| ) \, 
\rho \, d\rho,
\]
then $u^i$ and $u_M^i$ are identical, once $|\bx-(0,h)| >
M_0/4$. This is easily established by observing that for $|x|>M_0/4$,
the multipole expansion induced by the smooth, compactly supported source 
and the point source are identical.
The choice of $B$ follows directly from the Graf addition theorem \cite{nist}.
Since $1-W_{M_0}$ is identically zero in the interval
$[-M_0/4,M_0/4]$, it is clear that the product $(1-
W_{M_0})(u^i-u_{M_0}^i)$ is identically zero on the real axis. The
result now follows from Lemma \ref{lema1} and the fact that
$\reallywidehat{u_{M_0}^i}(\lambda)$ decays superalgebraically,
independent of $h$.
\end{proof}

\begin{mydef}
Let the functions $f,g: \mathbb{R} \rightarrow \mathbb{C}$.  The
functions $f$ and $g$ are said to be {\em spectrally equivalent} in the
far-field if $\reallywidehat{(f-g)(1-W)} = 0$.
\end{mydef}

Thus, Theorem \ref{dirthm} relies on the fact that $u^i$ and
$u_{M_0}^i$ are spectrally equivalent in the far field.

\section{The impedance Green's function}
\label{sec_impedance}

We now apply the preceding analysis to the case of the impedance
Green's function.
The impedance Green's function is the solution to the following
boundary value problem:
\begin{equation}\label{eq_impedgreens}
  \begin{aligned}
\Delta G_k^I(\bx,\bx_0) + k^2 G_k^I(\bx,\bx_0) &= \delta(\bx-\bx_0), &\qquad &y
> 0, \\
\frac{\partial G_k^I}{\partial n} +ik\alpha \, G_k^I &= 0, & &y =0,
  \end{aligned}
\end{equation}
where $\bx = (x,y)$ and $\bx_0 = (x_0,y_0)$, with both points
located in the upper half-plane.  The system~\eqref{eq_impedgreens}
can be re-written as a scattering problem if we express $G^I$ in two
parts:
\begin{equation}
G_k^I(\bx,\bx_0) = G_k(\bx,\bx_0) + u^s(\bx),
\end{equation}
i.e. the sum of a free-space Helmholtz point-source $G_k$ located at
$\bx_0$ and a scattered field (see~Figure~\ref{figintro}). Because of
translation invariance in $x$, we assume that the point source is
located at $\bx_0=(0,h)$ with $h$ small and positive.

Let us now denote by $\Omega$ the upper half space $y>0$ and
by $\Gamma=\partial\Omega$ the line $y=0$.
From~\eqref{helmholtz} and~\eqref{impedance}, the scattered field $u^s\in
C^2(\Omega)\cup C(\overline{\Omega})$ must satisfy 
\begin{equation}\label{scatt}
  \begin{aligned}
    \Delta u^s +k^2 u^s &= 0 &\qquad &\text{in }
    \Omega,\\ \frac{\partial u^s}{\partial n} +ik\alpha u^s &= g &
    &\text{on } \Gamma,
  \end{aligned}
\end{equation}
where $\bx=(x,0)$ on $\Gamma$ and
\begin{equation}
  g=-\frac{\partial G_k}{\partial n} -
  ik\alpha\,  G_k.
\end{equation}
The scattered field $u^s$ must also satisfy the radiation
condition
\begin{equation}\label{rad}
\lim_{r\rightarrow \infty}\sqrt{r}\left(\frac{\partial u^s}{\partial
  r}-iku^s\right)=0.
\end{equation}
The standard approach for evaluating $u^s$ is to write 
\[ u^s(\bx) = 
  \frac{1}{4\pi} \int_{-\infty}^\infty
  \frac{e^{-\sqrt{\lambda^2-k^2}y}}{\sqrt{\lambda^2-k^2}} \, 
e^{i\lambda x} \, \widehat{\xi}(\lambda) \, d\lambda,
\]
where $\widehat{\xi}(\lambda)$ is an unknown function.  This represents
a solution to the Helmholtz equation, enforces the desired radiation
condition, and permits the imposition of the boundary condition in
(\ref{scatt}) mode by mode. It is straightforward to check that
\[ 
 \widehat{\xi}(\lambda) 
 =
  -e^{-\sqrt{\lambda^2-k^2}h} 
\frac{\sqrt{\lambda^2-k^2}  + ik\alpha} 
{\sqrt{\lambda^2-k^2} - ik\alpha} \, .
\]
This is a convenient solution when $h$ is large.  When $h$ is small,
however, the interval of integration must clearly be of the order
$1/h$, which can be prohibitive.

While the standard approach to accelerating convergence of the
Sommerfeld integral is based on contour integration or some variant of
the method of images (including {\em complex images}), we
seek instead to compute $u^s$ using a combination of the Sommerfeld
integral and a single layer potential, as we did for the Dirichlet
problem above.  For this, we let
\begin{equation}\label{rep}
u^s = S_{\Gamma_0}[ \sigma_W ] + F_{I_0}[ \widehat{\xi}_W ],
\end{equation} 
where 
\begin{align}
S_{\Gamma_0}[\sigma](\bx) & = \int_{-M_0}^{M_0} G_k(\bx,\bx') \,
  \sigma(\bx') \, ds(\bx'), \\ F_{I_0}[\widehat{\xi}](\bx) & =
  \int_{-N_0}^{N_0}\frac{ e^{-\sqrt{\lambda^2-k^2}y}}
      {\sqrt{\lambda^2-k^2}}e^{i\lambda x} \, \widehat{\xi}(\lambda) \,
      d\lambda.
\end{align}
Here, $\Gamma_0$ is the finite segment $(-M_0,M_0)$ on the physical
interface $\Gamma$ and $I_0= (-N_0,N_0)$ is a finite segment in the
Fourier transform domain.  We assume $\sigma_W\in C({\Gamma_0})$ and
$\widehat{\xi}_W\in C(I_0)$. Note that $\sigma(\bx') = \sigma(x')$
with $\bx'=(x',0)$ on $\Gamma_0$.

Unlike the Dirichlet problem, we do not have an analytic formula
for the single layer density $\sigma$ and cannot determine 
$\sigma_W$ by a simple windowing procedure. Instead, we first determine
a density $\sigma$ on $\Gamma_0$ by letting
\[ u^s_1 = S_{\Gamma_0}[\sigma]  \]
and enforcing the boundary conditions in (\ref{scatt}) only on the
interval $[-M_0,M_0]$.  Using standard jump conditions
\cite{Cot2,guenther_lee}, this leads to the {\em local} integral
equation
\begin{equation}\label{finite_int}
-\frac{1}{2}\sigma+ ik\alpha \, 
S_{\Gamma_0} [\sigma] = g
\end{equation}
on $\Gamma_0$.
We will show that $\sigma \in C(\Gamma_0)$ for smooth right-hand-side
functions $g$.  After finding $\sigma$, we define $\sigma_W$ by
\begin{equation}\label{densitysig}
\sigma_W = W_{M_0} \, {\sigma},
\end{equation}
as above.

Given $\sigma_W$ from (\ref{densitysig}), we may substitute it
into~\eqref{rep} and solve for $\widehat{\xi}_W$ in the Fourier
domain:
\begin{equation}\label{freqequ}
\widehat{\xi}_W(\lambda) \left( -1 + \frac{ik
  \alpha}{\sqrt{\lambda^2-k^2}} \right) = e^{-\sqrt{\lambda^2-k^2}h}
\left (1 + \frac{ik\alpha}{\sqrt{\lambda^2-k^2}} \right) + \frac{1}{2}
\widehat{\sigma_W} - ik\alpha \widehat{S_{\Gamma_0}\sigma_W }
\end{equation}
for each $\lambda \in I_0$.
We discuss the well-posedness of the local integral equation
in Section~\ref{sec_wellposed} and
show in Section~\ref{sec_sommerfeld} that $\widehat{\xi}_W$ decays 
exponentially, in a manner controlled by $M_0$, the length of the window 
interval $\Gamma_0$, independent of the source location $h$.

\section{The layered media Green's function} \label{sec_layered}

The scheme described above for impedance boundary conditions can be
extended in a straightforward manner to the case of 
layered media. For simplicity, we assume there is a single material 
interface, $\Gamma$, that  
the wavenumber is $k_1$ in the
upper layer $\Omega_1 =\{ y>0\}$, and that the wavenumber
is $k_2$ in the lower layer
$\Omega_2 = \{y<0\}$. 
The layered media Green's function is the solution to the 
boundary value problem:
\begin{equation}\label{eq_lmgreens}
  \begin{aligned}
    \Delta G_{lm}(\bx,\bx_0) + k_1^2 G_{lm}(\bx,\bx_0) &=
    \delta(\bx-\bx_0),
    &\qquad &y>0,\\
    \Delta G_{lm}(\bx,\bx_0) + k_2^2 G_{lm}(\bx,\bx_0) &= 0, &\qquad &y<0,
  \end{aligned}
\end{equation}
subject to continuity conditions on $\Gamma$ of the form
\begin{equation}\label{continuity2}
\left[u \right] = 0, \qquad \left[ \frac{\partial u}{\partial n}
  \right] =0.
\end{equation}
We assume that the source lies in the upper half-space 
at $\bx_0 = (0,h)$, radiating at wavenumber $k_1$.
We then represent the total field in the top and bottom layers as
\begin{equation}
  \begin{aligned}
    u_1(\bx) &= u^s_1(\bx) + G_{k_1}(\bx,\bx_0), &\qquad
    &\text{for } \bx \in \Omega_1,\\ 
    u_2(\bx) &= u^s_2(\bx),
    & &\text{for } \bx \in \Omega_2. \\
  \end{aligned}
\end{equation}
By analogy with the impedance case,
we represent the scattered field in the form
\begin{equation}\label{rep_dielec}
\begin{aligned}
u^s_1 &=  S^{k_1}_{\Gamma_0}[\sigma_W] +D^{k_1}_{\Gamma_0}[\mu_W] +
F^{k_1}_{I_0}[\hat{\xi}_{W,1}],
&\qquad &\mbox{ for } \bx \in \Omega_1,\\
u^s_2 &=  S^{k_2}_{\Gamma_0}[\sigma_W] +D^{k_2}_{\Gamma_0}[\mu_W] +
F^{k_2}_{I_0}[\hat{\xi}_{W,2}],
& &\mbox{ for } \bx \in \Omega_2,
\end{aligned}
\end{equation}
where $\sigma$ and $\mu$ are unknown charge and dipole densities on
$\Gamma_0$. Note that the Sommerfeld densities $\hat{\xi}_{W,1}$ and 
$\hat{\xi}_{W,2}$ are distinct, one invoked for the upper layer and one
for the lower layer.
As above, we first solve a {\em local} integral equation on  
$\Gamma_0$ to obtain functions
$\sigma$ and $\mu$.
The integral equation is
derived by enforcing the continuity conditions~\eqref{continuity2},
and classical potential theory yields
\begin{equation}\label{finite_dielec}
\begin{aligned}
  \mu +  (S^{k_1}_{\Gamma_0}-S^{k_2}_{\Gamma_0})[\sigma]
   & = -G_{k_1}(\cdot,\bx_0), \\
 -\sigma + (T^{k_1}_{\Gamma_0}-T^{k_2}_{\Gamma_0})[\mu] & =
  -\frac{\partial G_{k_1}(\cdot,\bx_0)}{\partial n},
\end{aligned}
\end{equation}
where $T^k_{\Gamma_0}$ is the normal derivative of the 
double layer potential $D^k_{\Gamma_0}$ on
$\Gamma_0$, respectively \cite{Cot2,guenther_lee}.
Once equation \eqref{finite_dielec} is solved, we let
\begin{equation}
  \begin{aligned}
    \sigma_W & = W_{M_0}\, {\sigma}, \\
    \mu_W & = W_{M_0}\, {\mu}. 
  \end{aligned}
\end{equation}
We then substitute $\sigma_W$ and $\mu_W$ into the representation
\eqref{rep_dielec}. Taking the Fourier transform, we enforce
the continuity conditions \eqref{continuity2} frequency by frequency, and
obtain $\hat{\xi}_{W,1}$ and $\hat{\xi}_{W,2}$.

\section{Well-posedness of the integral equation}
\label{sec_wellposed}

Our method relies on the solvability of 
equations~\eqref{finite_int} and \eqref{finite_dielec}.

\begin{theorem}\label{thm1}
Let $g$ be a H\"{o}lder continuous function
on $\Gamma_0$ with exponent
$\alpha>0$, that is,  $g\in C^{0,\alpha}(\Gamma_0)$.
Then there exists a unique solution ${\sigma}\in
C(\Gamma_0)$ to the integral equation
\begin{equation}
-\frac{1}{2}{\sigma} + ik\alpha \, S_{\Gamma_0} [\sigma] = g. 
\end{equation}
\end{theorem} 

\begin{proof}
Since $S_{\Gamma_0} $ is a compact operator from $C(\Gamma_0)$ to
$C(\Gamma_0)$ \cite{Cot2}, we may obtain the desired result by 
means of the Fredholm alternative and simply show
uniqueness for the homogeneous case.
Thus, assume $g=0$ and let
\begin{equation}
  v(\bx) = S_{\Gamma_0} [\sigma](\bx) \qquad
  \mbox{ for } \bx\in \mathbb{R}^2\backslash \Gamma_0.
\end{equation}
A simple calculation shows that $v$ satisfies the boundary value problem
\begin{equation}\label{arcprob}
\begin{aligned}
  \Delta v + k^2v &= 0 &\qquad
  &\mbox{in }  \mathbb{R}^2\backslash \Gamma_0, \\
\frac{\partial v^+}{\partial n} + ik \alpha v^+ &= 0 & &\mbox{on } \Gamma_0, \\
\frac{\partial v^-}{\partial n} - ik \alpha v^- &= 0 & &\mbox{on } \Gamma_0,
\end{aligned}
\end{equation}
where 
\begin{equation}
\begin{aligned}
v^{\pm}(\bx) &= \lim_{\delta \rightarrow 0^+ } v(\bx\pm \delta n(\bx)), \\
\frac{\partial v^{\pm}(\bx)}{\partial n} &=
\lim_{\delta \rightarrow 0^+ } n(\bx) \cdot \nabla v(\bx \pm \delta n(\bx)),
\end{aligned}
\end{equation}
for $\bx\in \Gamma_0$.  The system \eqref{arcprob} defines an impedance
problem on the open arc $\Gamma_0$, which has a unique solution
\cite{Kress2003}. With zero boundary data, it has only the trivial
solution, so that 
\begin{equation}
  {\sigma}= \frac{\partial v^-}{\partial n} -
\frac{\partial v^+}{\partial n} = 0.
\end{equation}
\end{proof}

\begin{remark}\label{rem1}
It is important to note that in the preceding theorem, 
$\Gamma_0$ is an open interval. 
The density ${\sigma}$ generally exhibits singularities at
the end points \cite{Kress2003,Osipov1999}.
\end{remark}

\begin{remark}\label{rem2}
When $h>0$, $g\in C^{\infty}(\Gamma_0)$ in equation~\eqref{finite_int}
and $S_{\Gamma_0}$ is a pseudo-differential operator of order
minus-one~\cite{Mclean2000}. This implies that ${\sigma} = 2(ik\alpha \,
S_{\Gamma_0} {\sigma}-g)\in C^{0,\alpha}(\Gamma_0)$ with
$0<\alpha<1$. Thus, ${\sigma}\in C^{\infty}(\Gamma_0)$.
\end{remark}

\begin{remark}
The proof of existence and uniqueness for the local integral equation
in the layered media case is analogous and omitted.
\end{remark}

\section{Exponential decay of the Sommerfeld integral} 
\label{sec_sommerfeld}

In this section, we outline a proof of the fact that $\hat{\xi}_W$
in~\eqref{freqequ} is rapidly decaying, independent of the source
location $h$. The dielectric case is more involved but the proof
follows from the same reasoning.

\begin{theorem} \label{impthm}
Let $u^s$ be the scattered field induced by a point source at $(0,h)$
satisfying (\ref{scatt}). If we represent $u^s$ in the
form~\eqref{rep}, then 
$\left| \hat{\xi}_W(\lambda) \right|$ decays superalgebraically,
independent of $h$.
\end{theorem}

\begin{proof}[Sketch of proof]
Rewriting \eqref{freqequ} slightly, 
$\hat{\xi}_W$ satisfies 
\begin{equation}\label{eqxi}
  \begin{aligned}
\left(-1+\frac{ik\alpha}{\sqrt{\lambda^2-k^2}}\right)\hat{\xi}_W &=
\widehat{g} + \frac{1}{2}\widehat{\sigma_W} - ik\alpha \,
\reallywidehat{S_{\Gamma_0} [\sigma_W]} \\
&= \widehat{g} - \widehat{g_W} +
(\widehat{g_W} + \frac{1}{2}\widehat{\sigma_W} - ik\alpha \,
\reallywidehat{S_{\Gamma_0} [\sigma_W]}) \\
&= \widehat{g} - \widehat{g_W} +
\reallywidehat{W \cdot (g + \frac{1}{2}{\sigma} - ik\alpha \,
  {S_{\Gamma_0} [\sigma]})} 
 + ik\alpha \, \reallywidehat{W \cdot
  {S_{\Gamma_0} [(1-W) \sigma]}} \\
&\qquad - ik\alpha
\reallywidehat{(1- W) {S_{\Gamma_0} [\sigma_W]}} \, .
  \end{aligned}
\end{equation}
Each of these terms satisfies the desired superalgebraic decay estimate.
The result for $\widehat{g} - \widehat{g_W}$ follows the proof of
Theorem~\ref{dirthm}. The second term vanishes 
since it is the Fourier transform of the residual from solving the local 
integral equation.  (In practice, it is zero to discretization error.)
The third term is complicated 
since the solution of the local integral equation, $\sigma$, is
singular at the endpoints of $\Gamma_0$. However, the frequency
content of $S_{\Gamma_0} [(1-W) \sigma]$ is controlled near the
origin, and multiplication by $W = W_{M_0}$ restricts the function to the
interval $[-3M_0/4,3M_0/4]$. 
The last term follows from the spectral equivalence in the far field
of $S_{\Gamma_0} [\sigma_W]$ and a mollified version of
$\sigma_W$. 
\end{proof}
We state without proof the analogous result
for the dielectric case.

\begin{theorem} \label{lmthm}
Let $u^1$ and $u^2$ denote the total fields for $y>0$ and $y<0$,
respectively, satisfying the continuity conditions
(\ref{continuity2}), with a point source in the upper medium at
$(0,h)$.  If we represent $u^s$ in the form~\eqref{rep_dielec}, then
$\left| \hat{\xi}_{W,1}(\lambda) \right|$ and
$\left| \hat{\xi}_{W,2}(\lambda) \right|$ decay superalgebraically,
independent of $h$.
\end{theorem}

\section{Numerical examples}
\label{sec_numerical}

In the experiments described below, all of the local integral
equations are solved using Nystr\"om discretization on panels, each
containing (scaled) $16^\text{th}$-order Legendre nodes. We use
generalized Gaussian quadrature for the self-interaction panels,
following the approach described in~\cite{bremer, hao_barnett}.
Off-panel evaluation is carried out using Quadrature By
Expansion (QBX)~\cite{barnett-2014, klockner_2013, klockner_2013b} and the
resulting linear systems are solved iteratively using
GMRES~\cite{saad-1986}.

We fix the parameters $M_0=20$ and $N_0 = 30$ for all examples, and as
mentioned before, use the window function in~\eqref{err}. To
accurately evaluate the Sommerfeld integral and avoid the square root
singularity in the integrand, we
deform the integration contour along a hyperbolic tangent curve:
\begin{equation}\label{hypertan}
\lambda(t) = t-\frac{\tanh(t)}{2}i, \qquad t \in [-30,30].
\end{equation} 
This contour is then discretized using $600$ uniformly distributed
points in $t$, and the integral is evaluated by mean of the trapezoidal
rule (which will converge exponentially fast in this case).

\subsection{Impedance Green's function evaluation}

Our first example is simply the evaluation of the impedance Green's
function for a source located very near to the interface.  We set the
impedance constant $\alpha$ in equation~\eqref{impedance} to be
$\alpha = 1-0.1i$.  Numerical results are shown in
Table~\ref{tab:NumericalResult1} and Figure~\ref{Impedgreen}.

Figure~\ref{Impedgreen} plots the result when the source is located at
$(0,10^{-8})$ for wavenumber $k=10$. One can see a sharp spike in the
charge density $\sigma$ which causes slow convergence in the Fourier
domain. However, once we apply the window function to the local
integral equation, the resulting Sommerfeld density has a rapidly
decaying Fourier transform, as shown in Figure~\ref{Impedgreen3}. By
the time $\lambda = 30-\frac{\tanh(30)}{2}i$, $\hat\xi_W$ is already
less than $10^{-10}$. The density $\sigma$ was discretized along
$\Gamma_0$ using 
$160$  panels, adaptively refined toward the
origin. This required $2560$ discretization nodes.

Table \ref{tab:NumericalResult1} shows the value of the density
$\hat\xi_W$ at $t=-30$ in \eqref{hypertan} and the error in the
impedance condition $\partial u/\partial n +ik\alpha u$ at $(2.5,0)$
for a variety of wavenumbers $k$ and heights $h$.  Note that the error
is independent of the height $h$ and that the density $\hat{\xi}$ has
decayed rapidly.

\begin{remark}
When the source point is very close to the interface, such as
$h=10^{-8}$, the maximal normal derivative of $G_k$ on $\Gamma$ is of
the order $O(1/h)$, while the potential is of the order $O(\log{h})$.
To avoid catastrophic cancellation, we place an image source at the
reflected point across $\Gamma$ with the same strength as the original
source. This cancels the normal derivative of $G_k$ on $\Gamma$. It is
added back in the final evaluation.  We apply the same technique when
evaluating the layered media Green's function in the next
example. This is a finite precision issue, and independent of our
theory. For a general scattering problem, if the net contribution from
a continuous charge density on the interface is $O(1)$,
as in Example 3, we do not need to carry out this stabilization.
\end{remark}

\begin{figure}[t]
    \centering
    \begin{subfigure}[t]{.25\linewidth}
        \centering
        \includegraphics[width=1\linewidth]{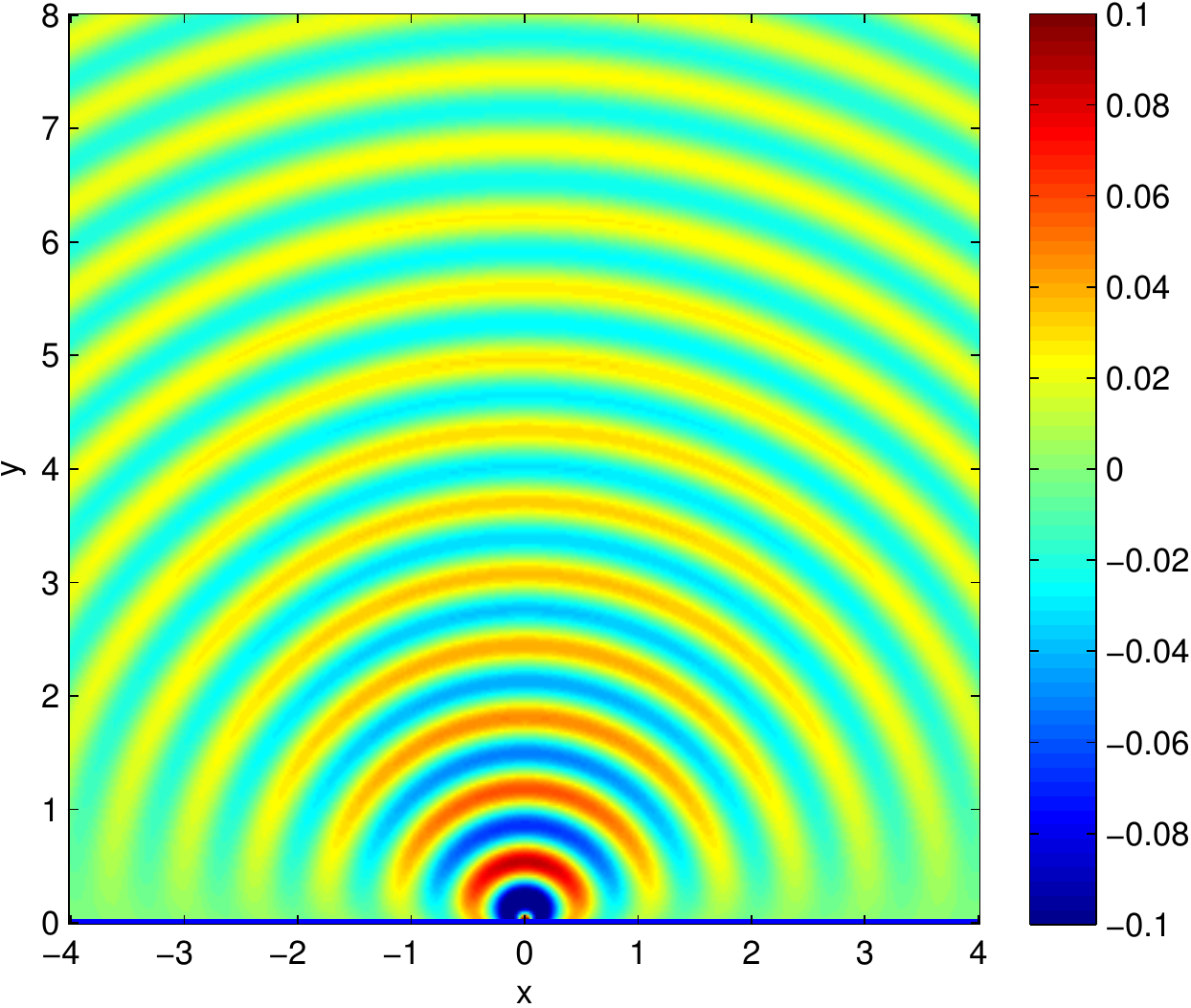}
        \caption{The real part of the impedance Green's function.}
        \label{Impedgreen1}
    \end{subfigure}
    \quad
    \begin{subfigure}[t]{.25\linewidth}
        \centering
        \includegraphics[width=1\linewidth]{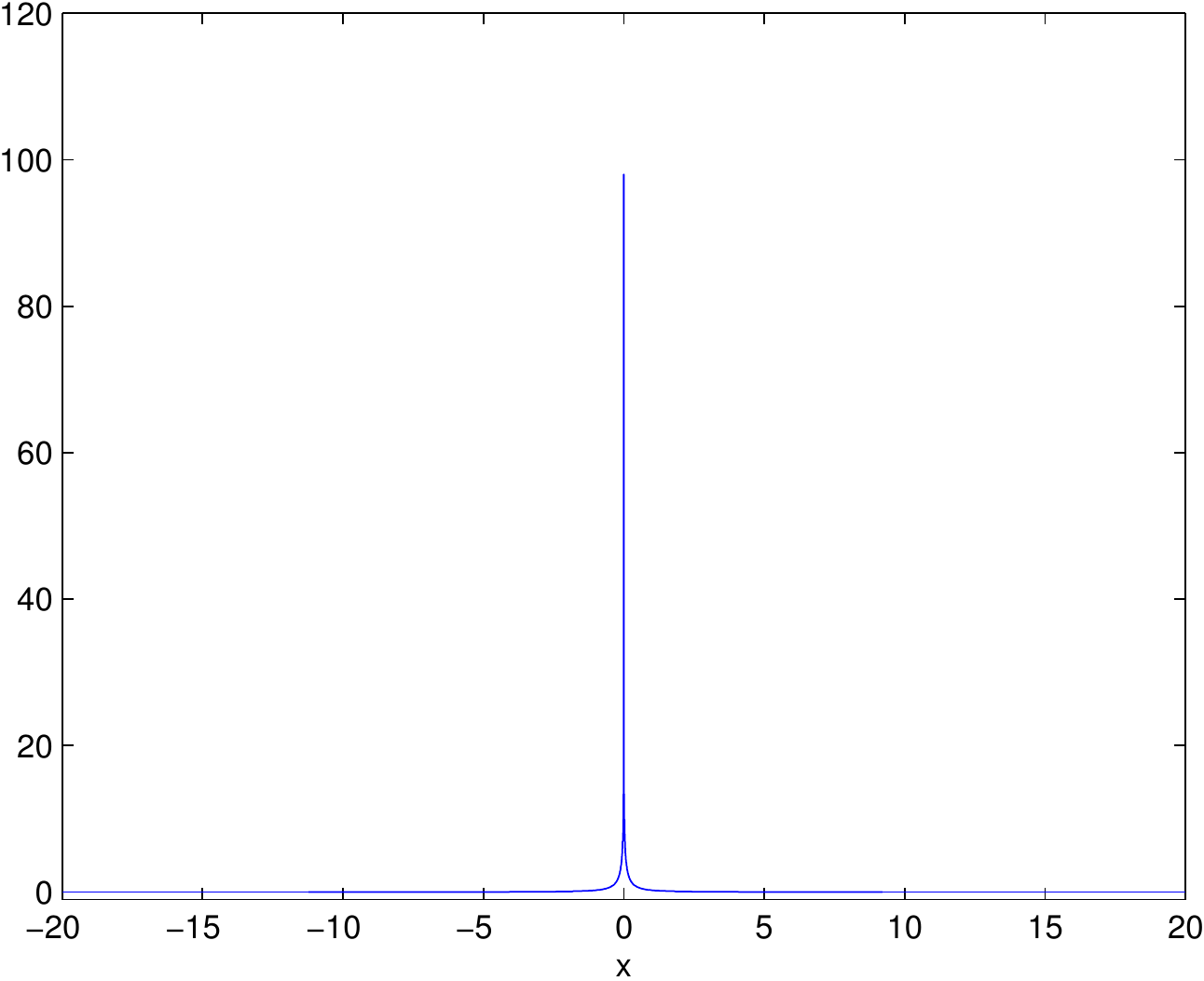}
        \caption{The magnitude of $\sigma$.}\label{Impedgreen2}
    \end{subfigure}
    \quad
    \begin{subfigure}[t]{.25\linewidth}
        \centering
        \includegraphics[width=1\linewidth]{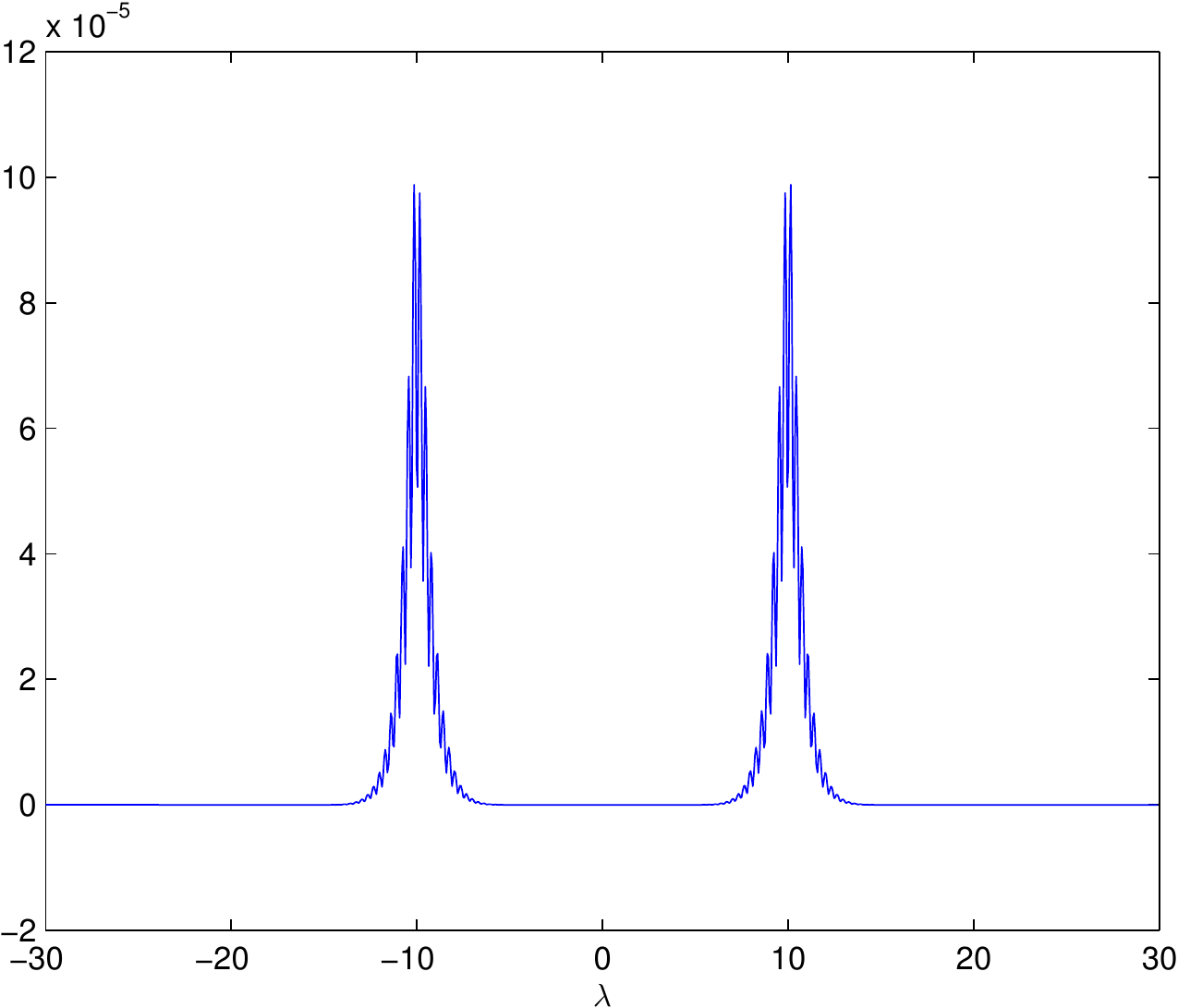}
        \caption{The magnitude of $\hat{\xi}$.}\label{Impedgreen3}
    \end{subfigure}
    \caption{The impedance Green's function, along with magnitudes of
      the layer potential density $\sigma$ and the Sommerfeld integral
      density $\hat\xi_W$.}
    \label{Impedgreen}
\end{figure}

\begin{table}[t]
\centering
  \begin{tabular}{|c|cccccc|}
\hline
Wavenumber $k$ & 1 & 1 & 10 & 10 & 20+$i$ & 20+$i$ \\
\hline
Height $h$ & $1$e-5 & $1$e-8 & $1$e-5 & $1$e-8 &$1$e-5 & $1$e-8 \\
\hline
$\hat{\xi}(-30)$ & $7.99$e-16 & $2.15$e-16 &$2.06$e-11 &$2.06$e-11 &$4.83$e-11 &$4.83$e-11 \\
\hline
Error  & $3.49$e-15 & $3.28$e-15  &$2.15$e-11 &$2.15$e-11 &$1.22$e-10 &$1.22$e-10 \\
\hline
\end{tabular}
\caption{Convergence results for the evaluation of the 
impedance Green's function when the source point is close to the interface.}
\label{tab:NumericalResult1}
\end{table}

\subsection{Layered media Green's function evaluation}

In our second numerical example, we evaluate the layered media Green's
function.  Figure~\ref{Dielecgreen} shows the real part of the Green's
function for $k_1=10$ in the upper half-space and $k_2=20$ in the
lower half-space with source point at $(0,10^{-8})$.  As in the
impedance example, there is a spike in the dipole density $\mu$ at
$x=0$. Using {\em only} a Sommerfeld integral approach would require a
prohibitively large interval of integral to obtain convergence, while
our hybrid scheme achieves rapid convergence in the Fourier domain, as
can be seen from the plot of $\hat{\xi}_{W,1}$ in
Figure~\ref{fig_dgreen3}.  More detailed data concerning errors in
evaluating the Green's functions are provided in
Table~\ref{tab:NumericalResult2}. The error is measured as the
discrepancy in the potential at $(2.5,0)$ by evaluating the limits
from above and below. The Sommerfeld density decays rapidly,
independent of the height $h$, consistent with our analysis. As with
the impedance problem, $160$ panels adaptively refined toward the
origin ($2560$ points) were required for the discretization of
$\sigma$ and $\mu$ on $\Gamma_0$.

\begin{figure}[t]
    \centering
    \begin{subfigure}[t]{.25\linewidth}
        \centering
        \includegraphics[width=1\linewidth]{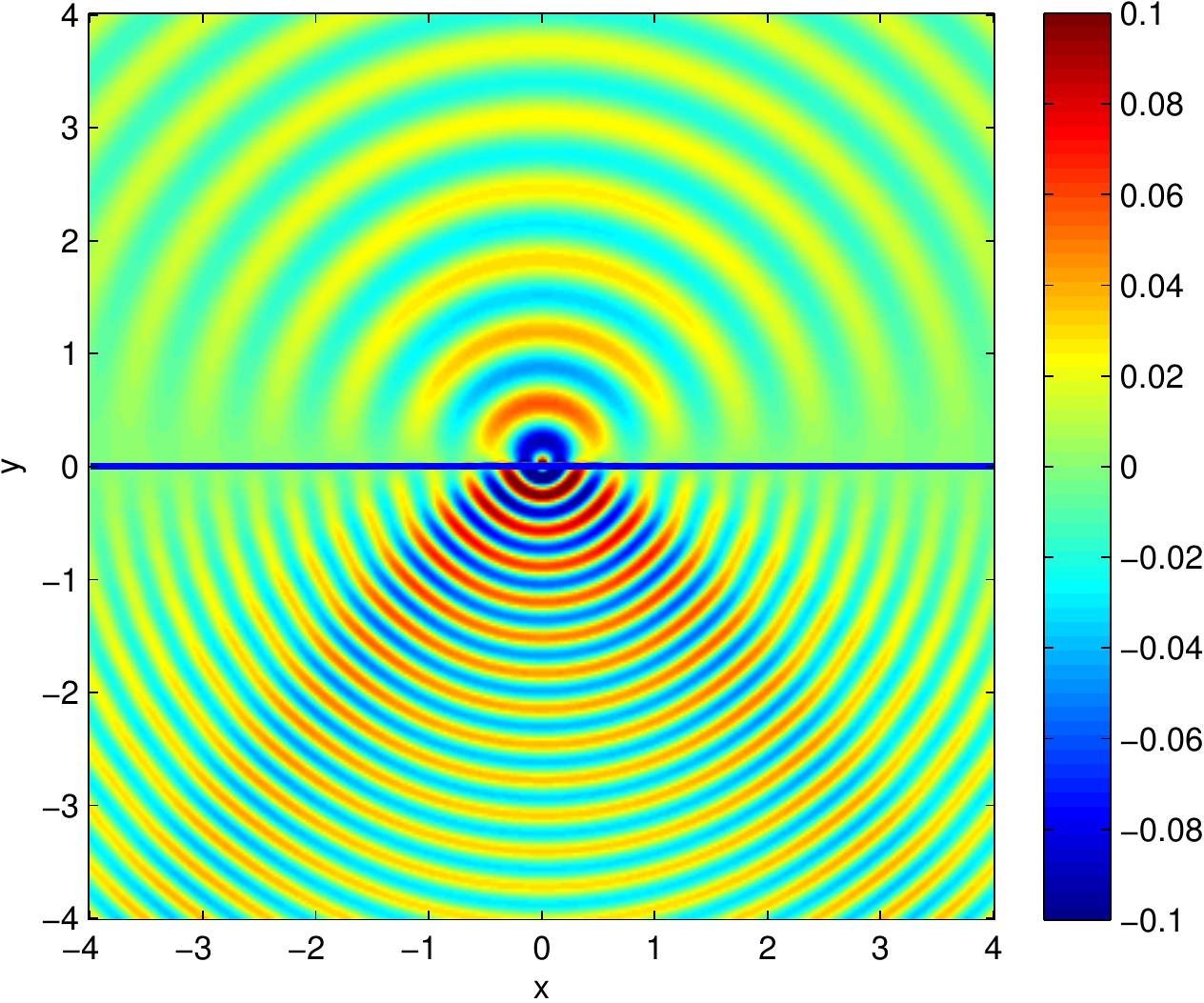}
        \caption{The real part of the layered media Green's function.}
        \label{fig_dgreen1}
    \end{subfigure}
    \quad
    \begin{subfigure}[t]{.25\linewidth}
        \centering
        \includegraphics[width=1\linewidth]{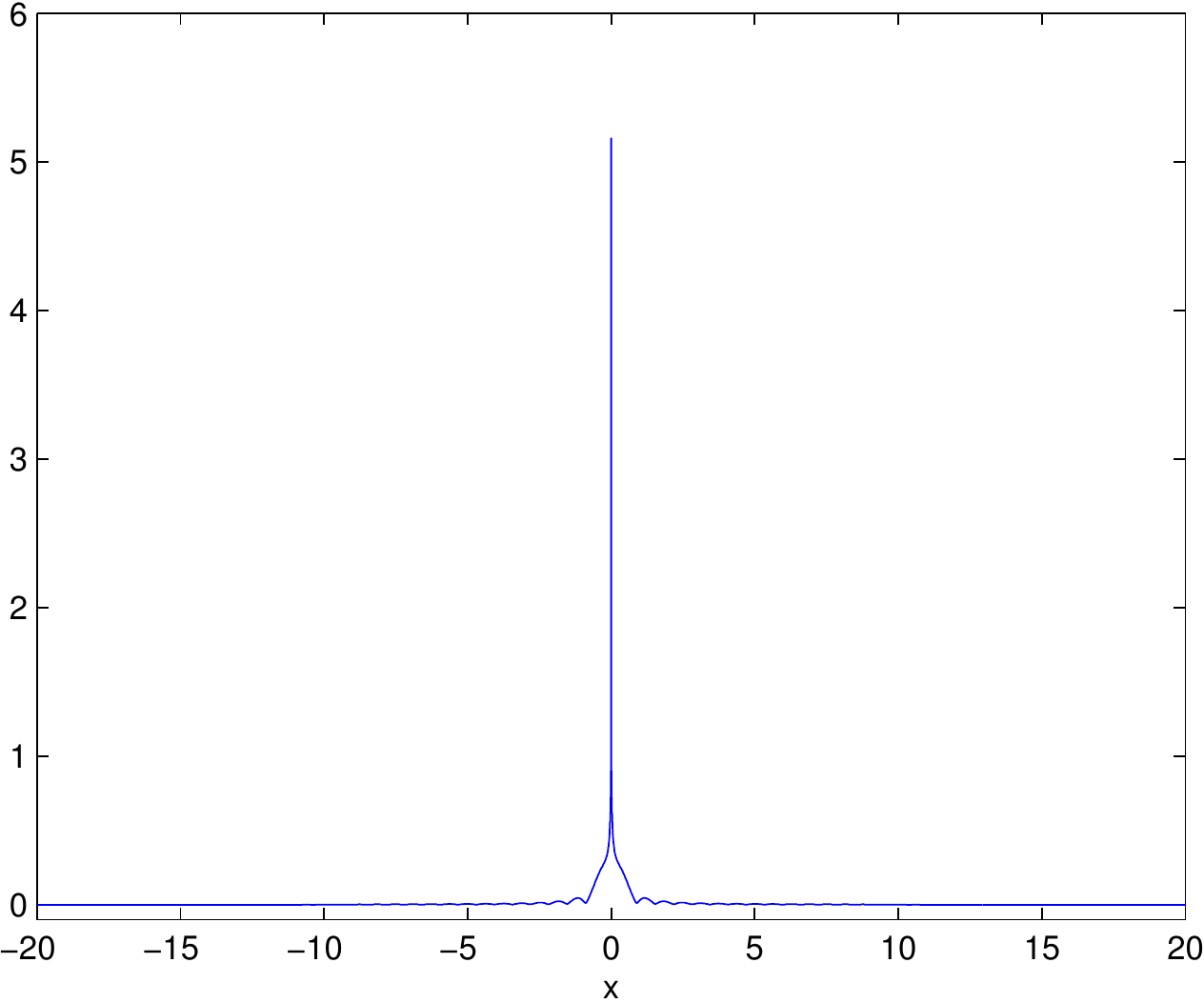}
        \caption{The magnitude of 
          $\mu$.}\label{fig_dgreen2}
    \end{subfigure}
    \quad
    \begin{subfigure}[t]{.25\linewidth}
        \centering
        \includegraphics[width=1\linewidth]{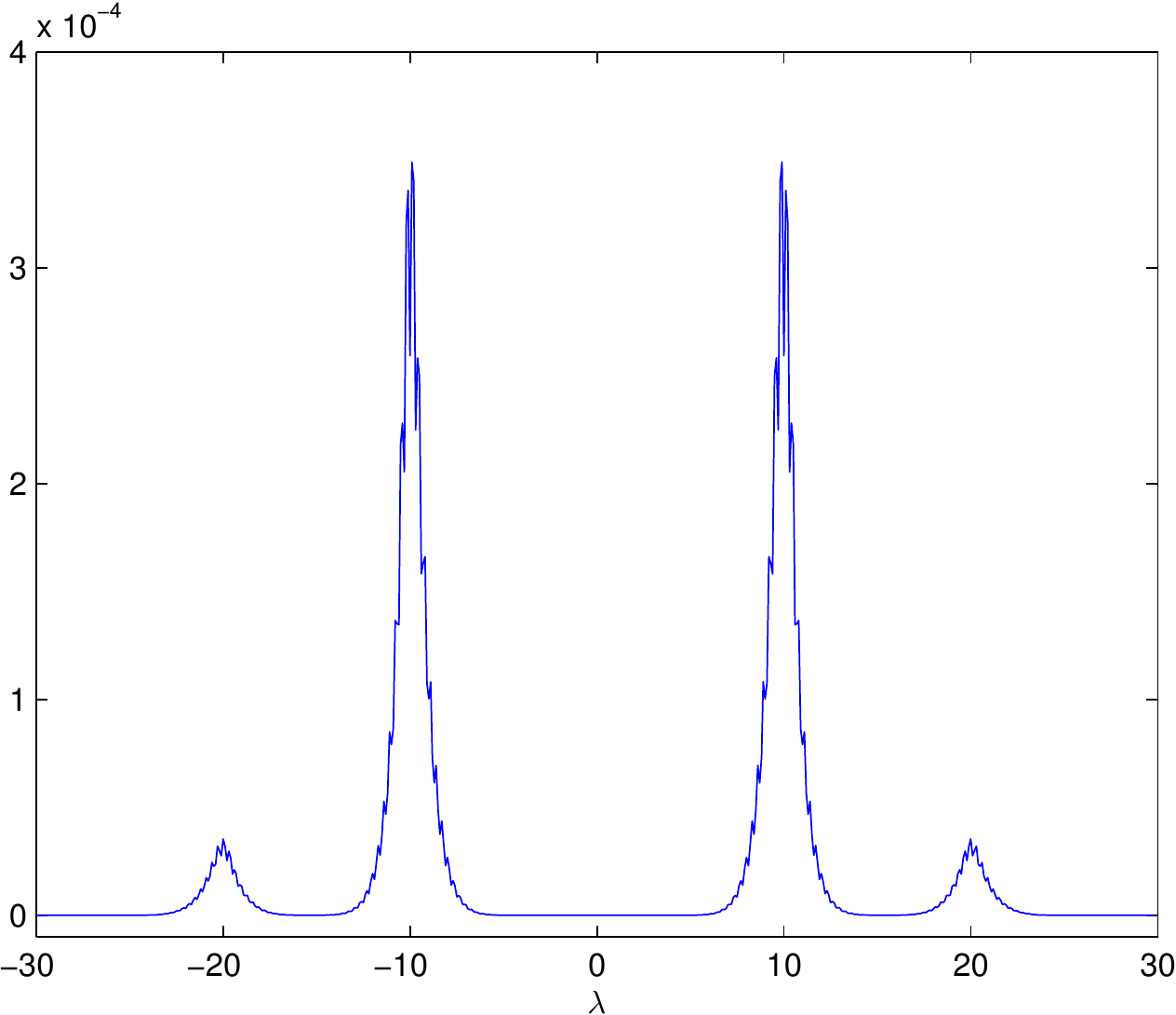}
        \caption{The magnitude of 
          $\hat{\xi_{W,1}}$.}\label{fig_dgreen3}
    \end{subfigure}
    \caption{The layered media Green's function, the magnitudes of
      the layer potential density $\mu$ and the Sommerfeld integral
      density $\hat\xi_{W,1}$.}
    \label{Dielecgreen}
\end{figure}

\begin{table}[t] 
  \begin{center}
    \small
\begin{tabular}{|c|cccccc|}
\hline
Wavenumber $k_1$ & 1 & 1 & 10+$i$ & 10+$i$ & 10 & 10 \\
\hline
Wavenumber $k_2$ & 5 & 5 & 5 & 5 & 20 & 20 \\
\hline
Height $h$ & $1$e-5 & $1$e-8 & $1$e-5 & $1$e-8 &$1$e-5 & $1$e-8 \\
\hline
$\hat{\xi_1}(-30)$  & $4.09$e-15 & $3.67$e-15 & $1.02$e-15 &$1.82$e-15 &$1.63$e-11 &$1.63$e-11 \\
\hline
Error  & $1.44$e-15 & $1.01$e-15 & $3.88$e-14 &$3.85$e-14 &$8.28$e-10 &$8.28$e-10 \\
\hline
\end{tabular}
\end{center}
\caption{Convergence results for the evaluation of the layered media 
Green's function when the source point is close to the interface.}
\label{tab:NumericalResult2}
\end{table}

\subsection{Scattering from inclusions}

We now show the application of our scheme to solving a scattering
problem from an object extremely close to, or partially buried in, a 
layered media interface. The incident wave is assumed to be generated
by a point source located at $(3,3)$.

\subsubsection{Close-to-touching object lying in the upper half-space}

We consider a circular object $\Omega_0$ with radius $1$ centered at
$(0,1+\epsilon)$ above the interface. We set $\epsilon = 10^{-10}$, so
that the object is only a distance $10^{-10}$ from the interface
$\Gamma$.  The object is assumed to be a perfect conductor or a
sound-soft scatterer, both of which require that we impose zero
Dirichlet conditions on the object.  This yields the following
boundary value problem:
\begin{equation}\label{eq_close}
  \begin{aligned}
    \Delta u + k_1^2 u &= 0, &\qquad &\text{in } \Omega_1, \\
    \Delta u + k_2^2 u &= 0, &\qquad &\text{in } \Omega_2, \\
    u &= 0, & &\text{on } \partial\Omega_0, \\
    \left[ u \right] = \left[ \frac{\partial u}{\partial
        y} \right]
    &= 0, &  &y = 0,
  \end{aligned}
\end{equation}
along with a suitable decay condition at infinity.

To solve this scattering problem, we add an additional
unknown dipole density $\mu_0$ on the boundary of the inclusion
$\Omega_0$.
The scattered field $u^s$ is then represented by
\begin{equation}\label{rep_dielec1}
u^s = \left\{ 
\begin{array}{ll}
  S^{k_1}_{\Gamma_0}[\sigma_W] +D^{k_1}_{\Gamma_0}[\mu_W] +
  F^{k_1}_{I_0}[\hat{\xi}_{W,1}] + D^{k_1}_{\partial \Omega_0}[\mu_0]
  &\qquad \text{in } \Omega_1,\\
  S^{k_2}_{\Gamma_0}[\sigma_W]
  +D^{k_2}_{\Gamma_0}[\mu_W] + F^{k_2}_{I_0}[\hat{\xi}_{W,2}] &\qquad \text{in }
  \Omega_2.
\end{array} \right.
\end{equation}
Results for $k_1=10$ and $k_2=20$ are shown in Figure~\ref{Dielec1}.
We only required $120$ panels adaptively refined toward the origin
($1920$ points) for the discretization of $\sigma$ and $\mu$ on
$\Gamma_0$.  The discretization of $\mu_0$ on $\partial \Omega_0$
required $30$ panels.

\begin{figure}[t]
  \centering
  \begin{subfigure}[t]{.4\linewidth}
    \centering
    \includegraphics[width=1\linewidth]{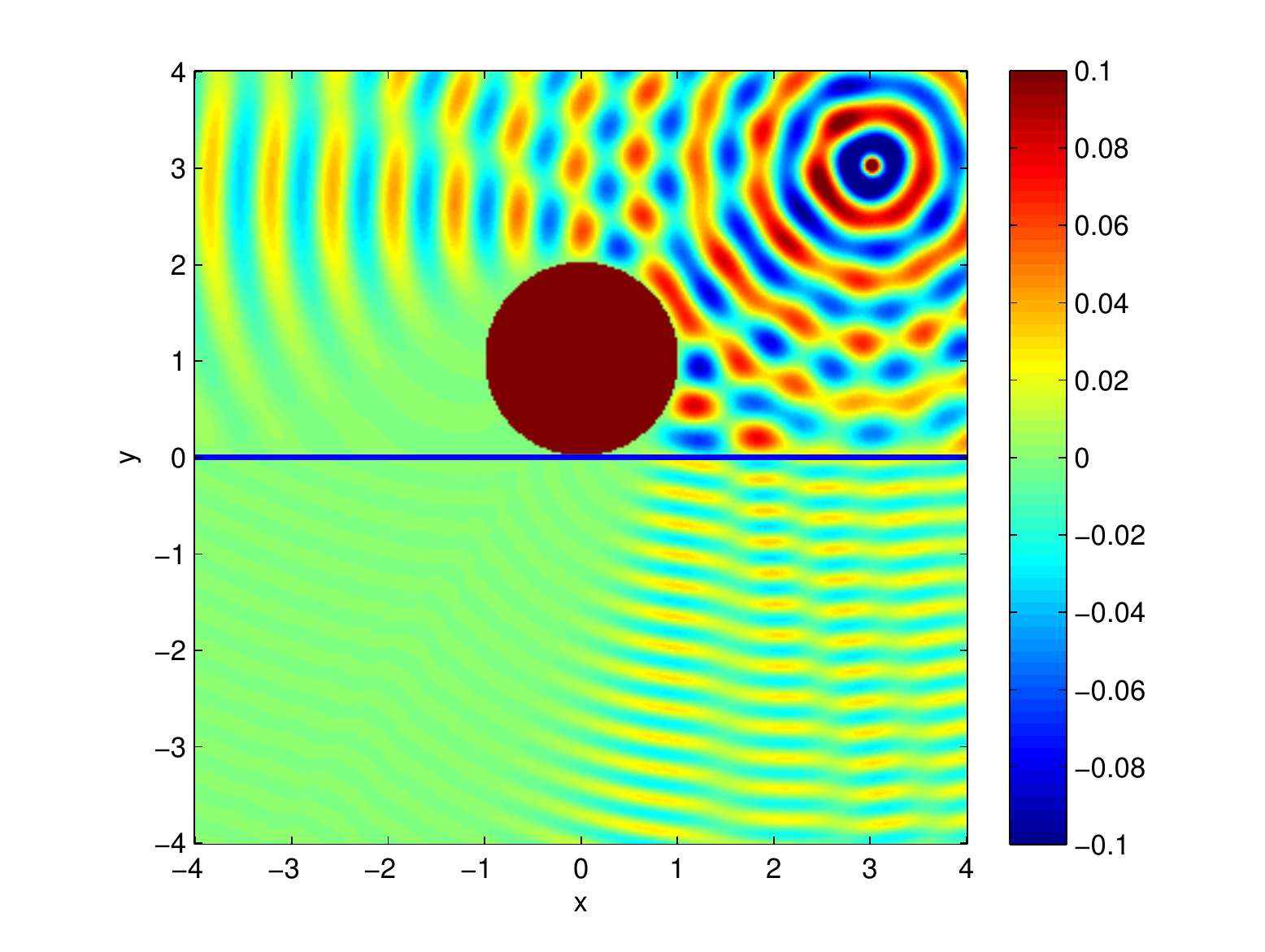}
    \caption{The real part of the total field scattered in the layered
      media with a circular object located right above the interface.}
    \label{fig_dielec1a}
  \end{subfigure}
  \qquad
  \begin{subfigure}[t]{.4\linewidth}
    \centering
    \includegraphics[width=1\linewidth]{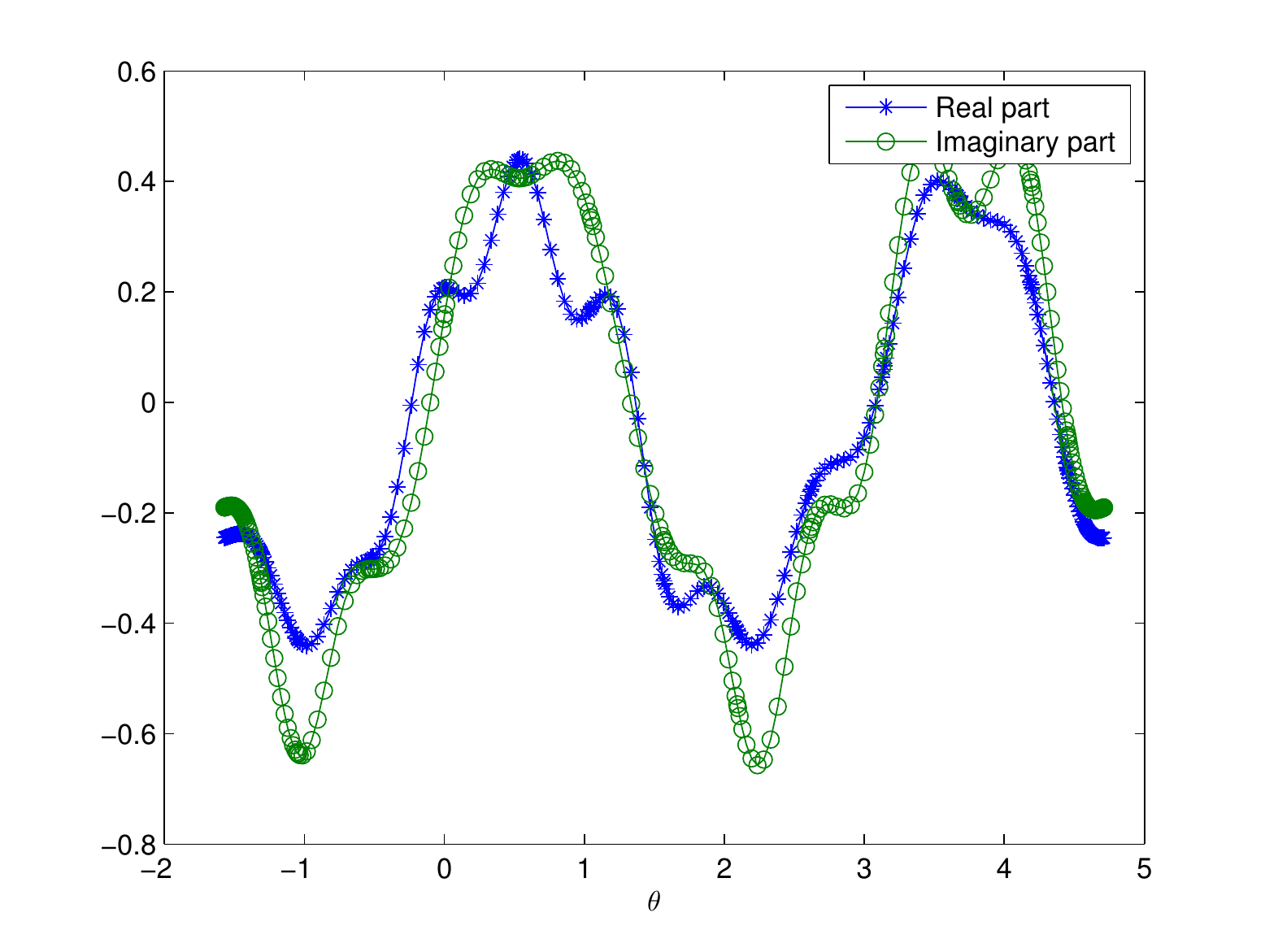}    
    \caption{The real and imaginary parts of the
      dipole density $\mu_0$ on the boundary of the object}
    \label{fig_dielec1b}
  \end{subfigure}
  \\
  \begin{subfigure}[t]{.4\linewidth}
    \centering
    \includegraphics[width=1\linewidth]{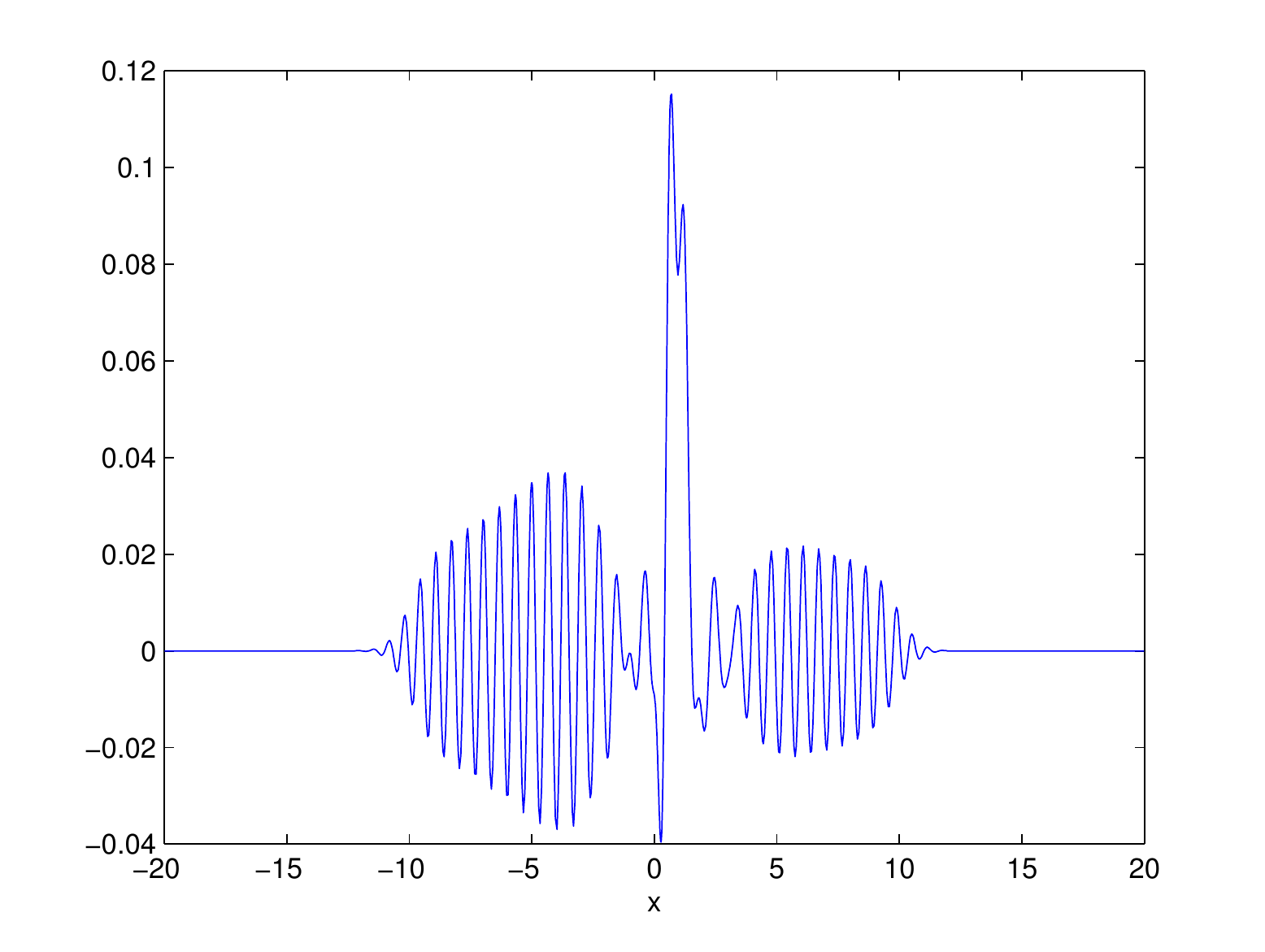}    
    \caption{Real part of the charge density $\sigma_W$.}
    \label{fig_dielec1c}
  \end{subfigure}
  \qquad
  \begin{subfigure}[t]{.4\linewidth}
    \centering
    \includegraphics[width=1\linewidth]{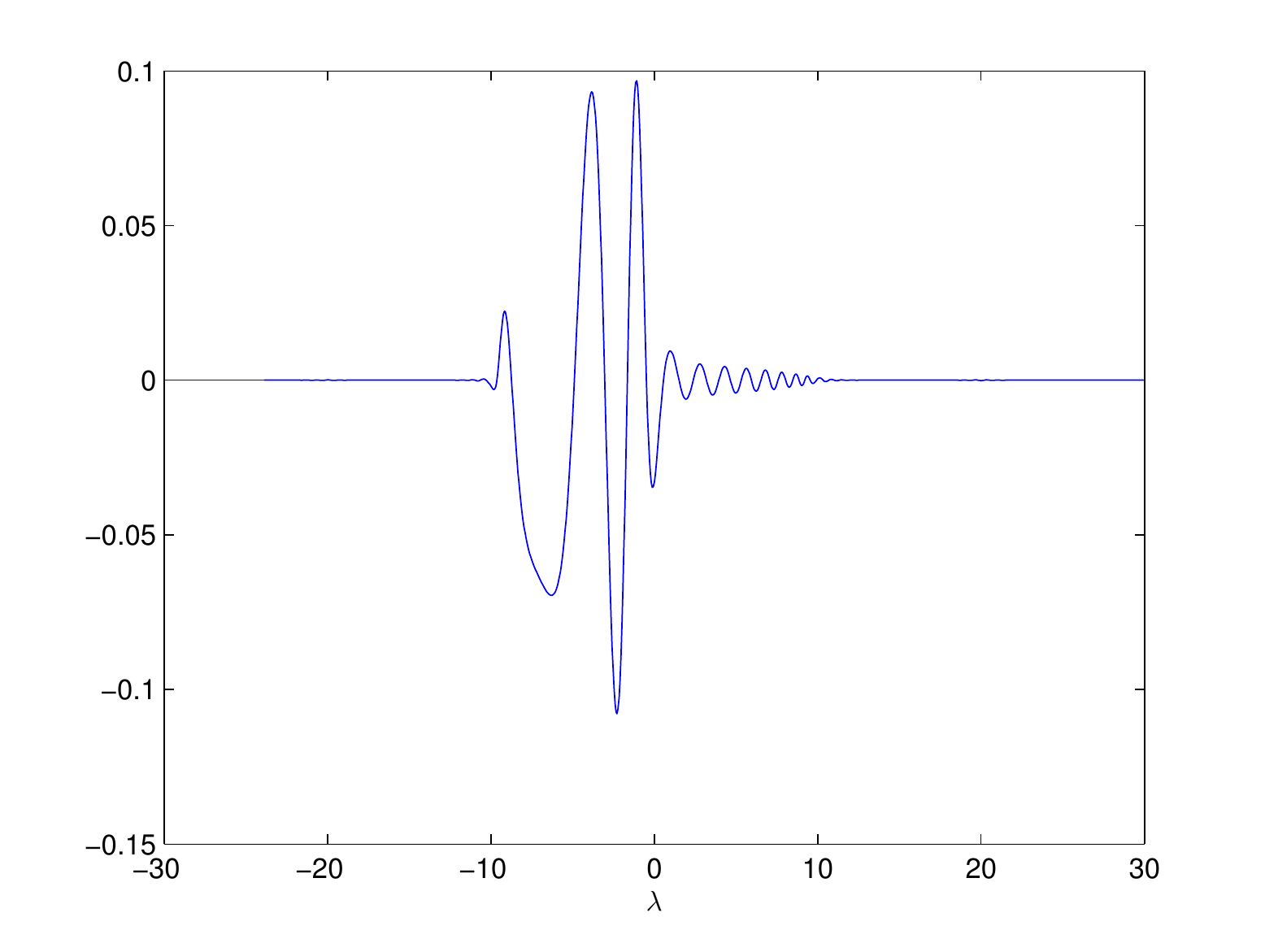}
    \caption{Real part of the Sommerfeld density $\hat{\xi}_{W,1}$.}
    \label{fig_dielec1d}
  \end{subfigure}
  \caption{Scattering from an object arbitrarily close to a layered
    media interface.}
  \label{Dielec1}
\end{figure}

\begin{table}[t] 
\begin{center}
    \small
\begin{tabular}{|c|cccccc|}
\hline
Wavenumber $k_1$ & 1 & 1 & 10+$i$ & 10+$i$ & 10 & 10 \\
\hline
Wavenumber $k_2$ & 2 & 5 & 5 & 5+$i$ & 20+5$i$ & 20 \\
\hline
GMRES iter. & 9 & 10 & 14 & 14 &31 & 31 \\
\hline
$\hat{\xi_1}(-30)$ & $3.15$e-13 & $1.08$e-13 & $3.54$e-14 &$3.62$e-14 &$1.78$e-8 &$4.06$e-8 \\
\hline
Error  & $9.24$e-12 & $1.64$e-12 & $1.32$e-10 &$2.12$e-10 &$1.44$e-11 &$8.04$e-11 \\
\hline
\end{tabular}
\end{center}
\caption{Convergence results for scattering from an object close to the 
interface of the layered media. The closest distance is $10^{-10}$.}
\label{tab:NumericalResult3}
\end{table}

From Figure~\ref{fig_dielec1b}, the charge density $\mu_0$ on the
boundary of the disk behaves rather benignly. However, the density on
the interface has a sharp feature near $x=0$, as
expected. Figure~\ref{fig_dielec1d} shows rapid decay in the
Sommerfeld integrand.  More detailed errors for various values of the
wavenumber in the layered media are shown in
Table~\ref{tab:NumericalResult3}.  The GMRES iterations are stopped when
the residual is less than $10^{-10}$.  The errors are obtained by
solving an artificial scattering problem with a known exact
solution. For this, the field in each domain is defined by a set of
free-space sources located in the complement of the domain (see
\cite{JL2014} for further details). In all of our tests, the
Sommerfeld integrand decays rapidly.

\subsubsection{Scattering from a partially buried object}

In our final example, a disk of radius $1$ crosses the interface
between the two layers. The center of the disk is $(0,0)$, and we 
solve the same boundary value problem as in~\eqref{eq_close}. The
scattered field is represented as:
\begin{equation}\label{rep_dielec2}
u^s = \left\{ 
\begin{array}{ll}
  S^{k_1}_{\Gamma_0}[\sigma_W] +D^{k_1}_{\Gamma_0}[\mu_W] +
  F^{k_1}_{I_0}[\hat{\xi}_{W,1}] + D^{k_1}_{\partial \Omega_0}[\mu_0]
  &\qquad \text{in } \Omega_1,\\
  S^{k_2}_{\Gamma_0}[\sigma_W]
  +D^{k_2}_{\Gamma_0}[\mu_W] + F^{k_2}_{I_0}[\hat{\xi}_{W,2}]
    + D^{k_2}_{\partial \Omega_0}[\mu_0] &\qquad \text{in }  \Omega_2.
\end{array} \right.
\end{equation}
Note that the density $\mu_0$ on the scatterer is used globally -- on
both sides of the layered media interface.  As in the earlier work
\cite{GHL2014,JL2014}, this has the advantage that the resulting
integral equation is of the second kind.  Results are shown in
Figure~\ref{Dielec2}. The densities $\hat{\xi}_{W,1}$ and
$\hat{\xi}_{W,2}$ in the respective Sommerfeld integrals are again
rapidly convergent. Along $\Gamma_0$, $180$ panels ($2880$ points)
adaptively refined toward the intersection of the inclusion and
$\Gamma_0$ were used to discretize $\sigma$ and $\mu$ on
$\Gamma_0$. Discretizing $\mu_0$ on $\partial \Omega_0$ required $60$
panels.  Detailed numerical errors are presented in
Table~\ref{tab:NumericalResult4}.  The results, again, are consistent
with our analysis.

\section{Conclusions}
\label{sec_conclusions}

We have constructed a hybrid approach to acoustic wave scattering in
layered media or in half-spaces with impedance boundary
conditions. Our approach retains the advantages of classical (physical
space) layer potentials in handling close-to-touching interactions and
the advantages of the Sommerfeld integral in representing smooth
interactions along infinite boundaries.  By solving a local integral
equation and subtracting its effect from the original boundary data, we have
shown that the remaining problem can be solved in the Fourier domain
with a rapidly convergent integrand. 

We have also shown that the hybrid representation is very convenient
when solving scattering problems with obstacles (including partial
buried obstacles).  High-order accurate results are easily obtained
without the direct construction of the Green's function. Instead, our
representation splits the problem into a free-space scattering problem
posed on obstacles and finite segments, plus a Sommerfeld correction
to enforce the boundary condition along the infinite interface.  We are
currently extending our method to axisymmetric and fully
three-dimensional problems in both acoustics and electromagnetics.

\begin{figure}[t]
  \centering
  \begin{subfigure}[t]{.4\linewidth}
    \centering
    \includegraphics[width=1\linewidth]{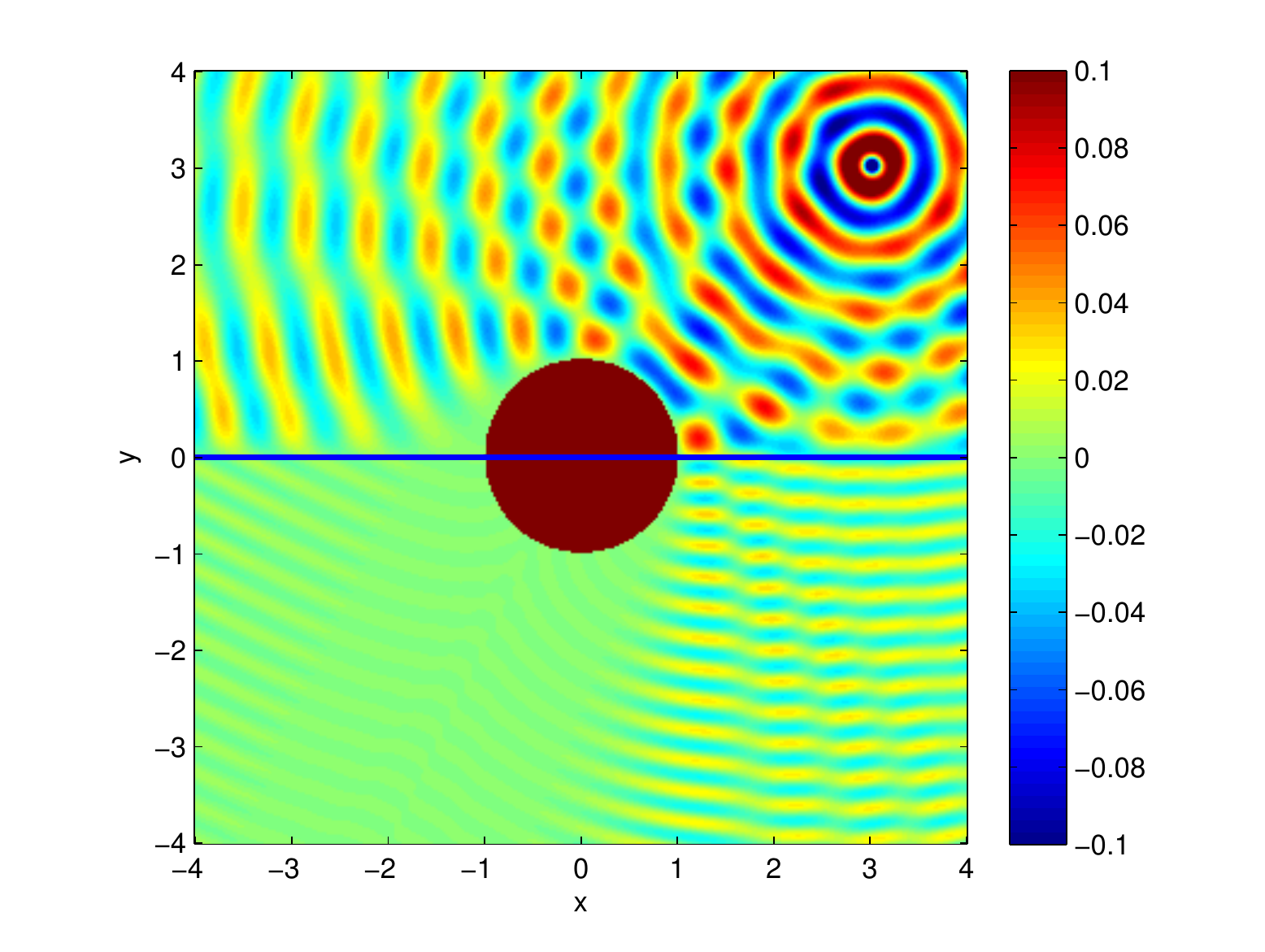}
    \caption{The real part of the total field scattered in the layered
      media with a circular object partially buried in the layers.}
    \label{fig_dielec2a}
  \end{subfigure}
  \qquad
  \begin{subfigure}[t]{.4\linewidth}
    \centering
    \includegraphics[width=1\linewidth]{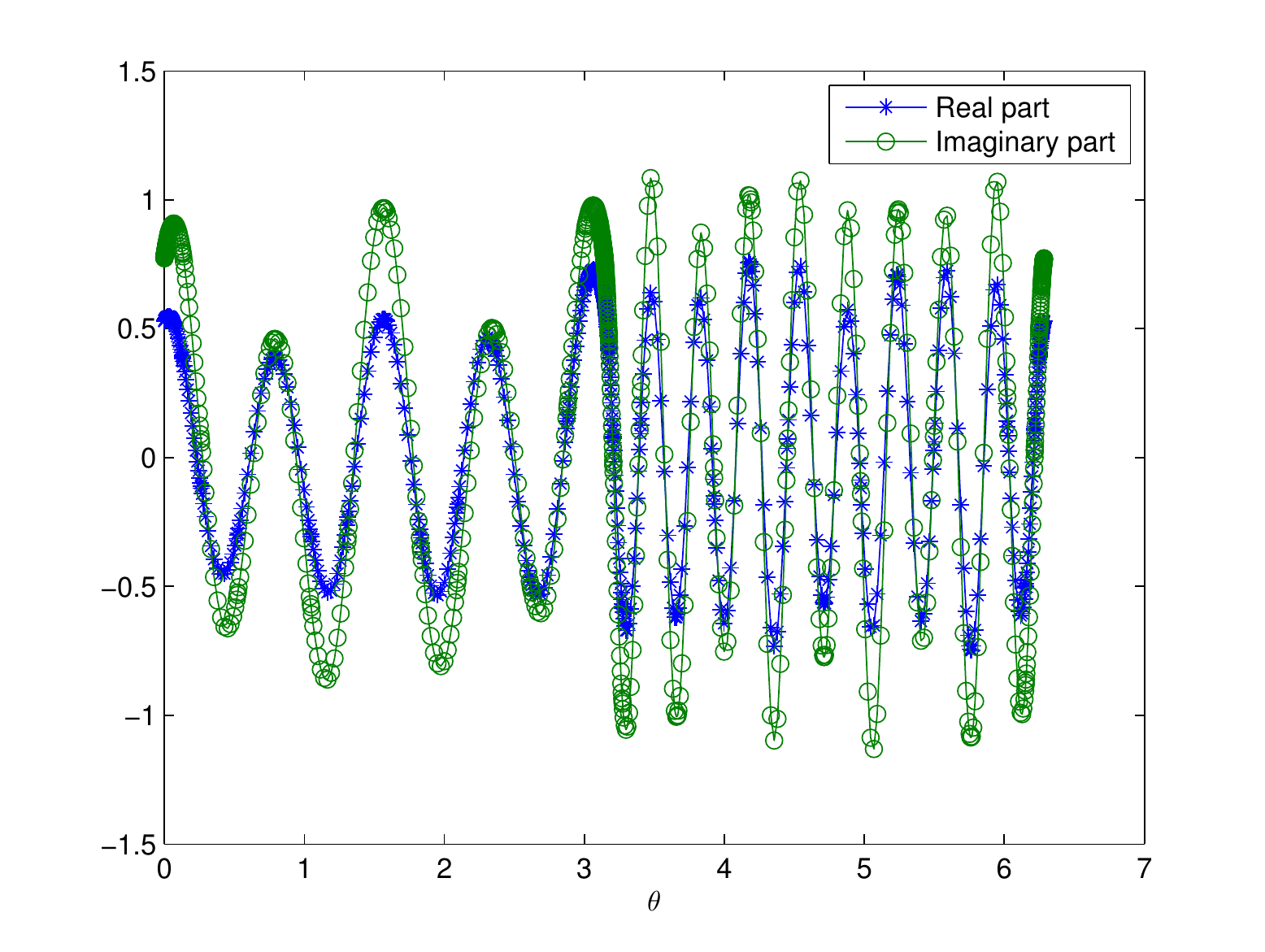}    
    \caption{The real and imaginary part of the
      dipole density $\mu_0$ on the boundary of the object}
    \label{fig_dielec2b}
  \end{subfigure}
  \\
  \begin{subfigure}[t]{.4\linewidth}
    \centering
    \includegraphics[width=1\linewidth]{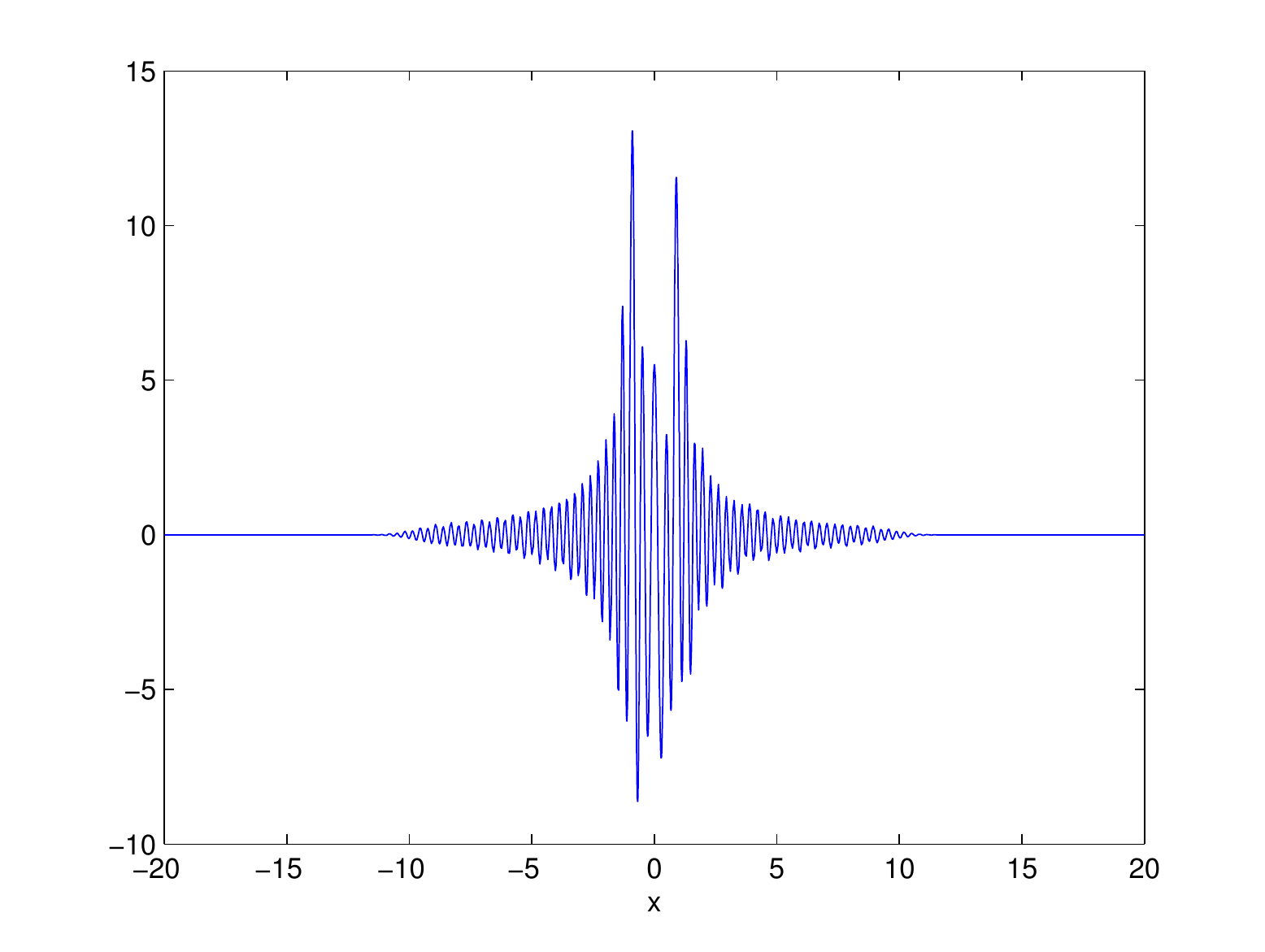}    
    \caption{Real part of the charge density $\sigma$.}
    \label{fig_dielec2c}
  \end{subfigure}
  \qquad
  \begin{subfigure}[t]{.4\linewidth}
    \centering
    \includegraphics[width=1\linewidth]{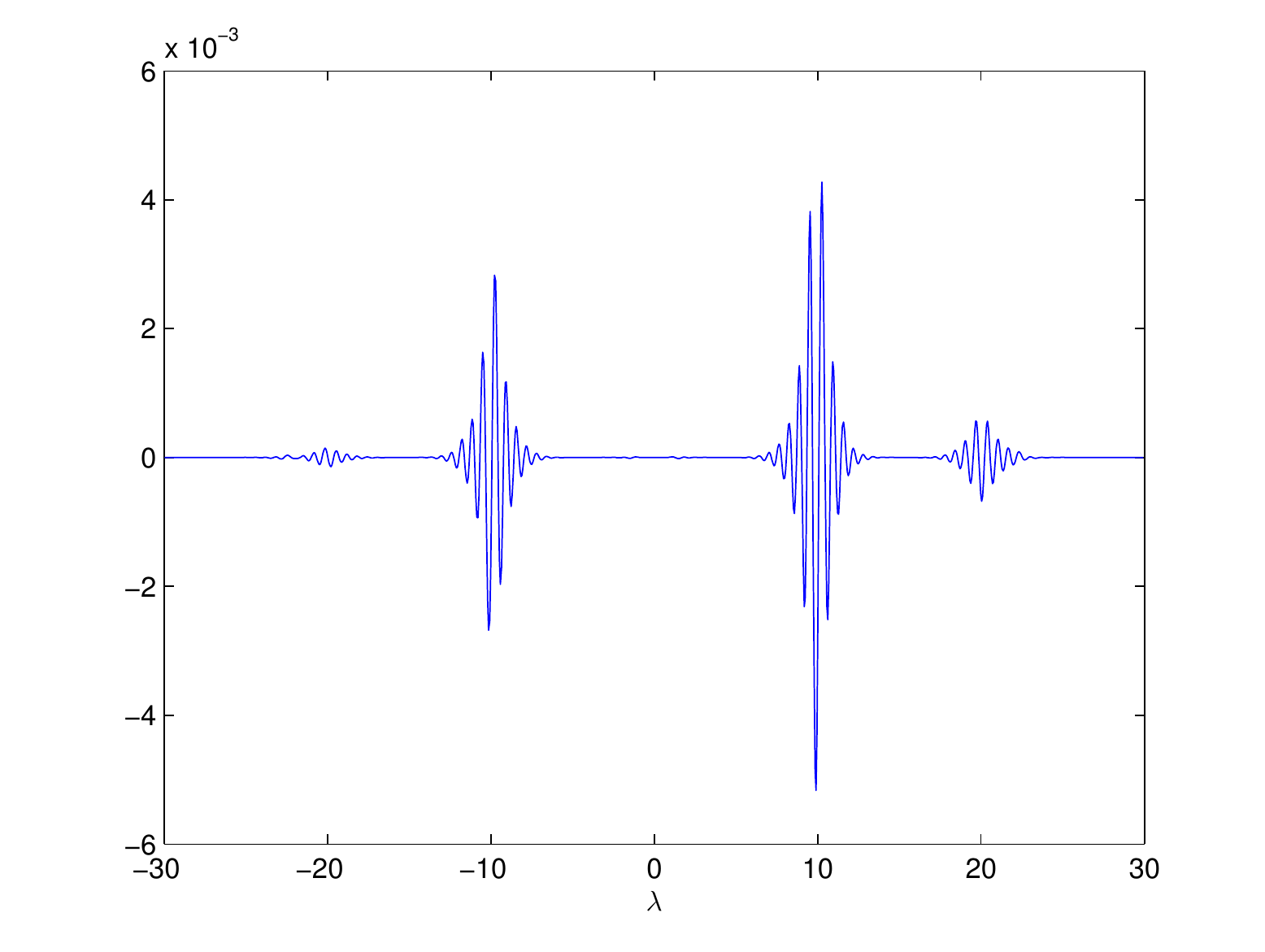}
    \caption{Real part of the Sommerfeld density $\hat{\xi}_1$.}
    \label{fig_dielec2d}
  \end{subfigure}
  \caption{Scattering from an object partially buried across a layered
    media interface.}
  \label{Dielec2}
\end{figure}

\begin{table}[t] 
\begin{center}
    \small
\begin{tabular}{|c|cccccc|}
\hline
Wavenumber $k_1$ & 1 & 1 & 10+$i$ & 10+$i$ & 10 & 10 \\
\hline
Wavenumber $k_2$ & 2 & 5 & 5 & 5+$i$ & 20+5$i$ & 20 \\
\hline
GMRES iter. & 10 & 13 & 16 & 12 &16 & 40 \\
\hline
$\hat{\xi_1}(-30)$  & $7.68$e-14 & $3.78$e-14 & $1.88$e-14 &$1.42$e-14 &$1.41$e-9 &$1.32$e-6 \\
\hline
Error  & $3.47$e-12 & $7.41$e-12 & $1.71$e-13 &$4.07$e-13 &$5.18$e-10 &$6.69$e-10 \\
\hline
\end{tabular}
\end{center}
\caption{Convergence results for scattering from a partially buried object 
in layered media.}
\label{tab:NumericalResult4}
\end{table}

\vspace{.2in}

\noindent
{\bf Acknowledgements:}
We would like to thank Alex Barnett, Charlie Epstein, and Tom Hagstrom 
for several useful discussions.

\bibliographystyle{abbrv}
\bibliography{reference_layer}

\end{document}